\documentclass[11pt, leqno]{amsart}
\usepackage{amsfonts,amssymb, amscd}
\usepackage{amsmath}
\usepackage{lmodern}
\usepackage{mathrsfs} 
\usepackage{amsthm}
\usepackage{qsymbols}
\usepackage{mathtools}
\mathtoolsset{showonlyrefs}
\usepackage{dsfont}
\usepackage{graphicx}
\usepackage{latexsym}
\usepackage{chngcntr}
\usepackage[noadjust]{cite}
\usepackage{paralist}
\usepackage[parfill]{parskip}
\usepackage{bm}
\usepackage{tikz-cd}
\usepackage{csquotes}
\usetikzlibrary{decorations.pathreplacing}

\usepackage[shortlabels]{enumitem}
\newtheorem{theorem}{Theorem}[section]
\newtheorem{lemma}[theorem]{Lemma}
\newtheorem{proposition}[theorem]{Proposition}
\newtheorem{corollary}[theorem]{Corollary}

\theoremstyle{definition}
\newtheorem{definition}[theorem]{Definition}
\newtheorem{remark}[theorem]{Remark}
\newtheorem{example}[theorem]{Example}
\usepackage[plainpages=false,pdfpagelabels,
citecolor=red]{hyperref}
\usepackage{xcolor}
\hypersetup{
 colorlinks,
 linkcolor={cyan!90!black},
 citecolor={magenta},
 urlcolor={green!40!black}
} 
\newcommand{\R}{\mathbb{R}}
\newcommand{\IC}{\mathbb{C}}
\newcommand{\IN}{\mathbb{N}}

\newcommand{\Res}{\mathcal{R}}
\newcommand{\Ext}{\mathcal{E}}
\newcommand{\Pro}{\mathcal{P}}
\newcommand{\Proo}{\mathcal{Q}}

\newcommand{\cM}{\mathcal{M}}
\newcommand{\cL}{\mathcal{L}}
\newcommand{\cH}{\mathcal{H}}

\newcommand{\cE}{\mathcal{E}}

\let\SS\S	
\renewcommand{\S}{\mathcal{S}}

\renewcommand{\L}{\mathrm{L}}
\newcommand{\C}{\mathrm{C}}
\newcommand{\B}{\mathrm{B}}
\renewcommand{\H}{\mathrm{H}}
\newcommand{\W}{\mathrm{W}}

\newcommand{\Sec}{\mathrm{S}}

\let\ii\i
\renewcommand{\i}{\mathrm{i}}
\renewcommand{\d}{\,\mathrm{d}}
\newcommand{\rad}{\mathrm{r}}

\newcommand{\len}{\ell}
\newcommand{\eps}{\varepsilon}

\renewcommand\Re{\operatorname{Re}}

\newcommand{\SP}{\, |\,}

\newcommand{\bO}{\boldsymbol{O}}
\newcommand{\bD}{\boldsymbol{D}}
\newcommand{\bN}{\boldsymbol{N}}
\newcommand{\bL}{\boldsymbol{L}}
\newcommand{\bA}{\boldsymbol{A}}
\newcommand{\bb}{\boldsymbol{b}}
\newcommand{\bc}{\boldsymbol{c}}
\newcommand{\bbd}{\boldsymbol{d}}
\newcommand{\bPiB}{\boldsymbol{\Pi_B}}
\newcommand{\bGamma}{{\boldsymbol{\Gamma}}}

\newcommand{\cl}[1]{\overline{#1}}
\renewcommand\div{\operatorname{div}}
\DeclareMathOperator{\bd}{\partial \!}
\DeclareMathOperator{\supp}{supp}

\DeclareMathOperator{\dist}{d}
\DeclareMathOperator{\diam}{diam}

\DeclareMathOperator{\Rg}{\mathsf{R}}
\DeclareMathOperator{\Ke}{\mathsf{N}}
\DeclareMathOperator{\dom}{\mathsf{D}}

\DeclareMathOperator{\codim}{codim}

\def\XXint#1#2#3{{\setbox0=\hbox{$#1{#2#3}{%
\int}$ }
\vcenter{\hbox{$#2#3$ }}\kern-.6\wd0}}
\title[The Kato square root problem on locally uniform domains] {The Kato square root problem on locally uniform domains}
\author{Sebastian Bechtel}
\author{Moritz Egert}
\author{Robert Haller-Dintelmann} 
\address{Fachbereich Mathematik, Technische Universit\"at Darmstadt, Schlossgartenstr. 7, 64289 Darmstadt, Germany}
\email{bechtel@mathematik.tu-darmstadt.de}
\email{haller@mathematik.tu-darmstadt.de}
\address{Universit\'e Paris-Saclay, CNRS, Laboratoire de Math\'{e}matiques d'Orsay, 91405 Orsay, France}
\email{moritz.egert@universite-paris-saclay.fr}
\subjclass[2010]{Primary: 47A60, 35J47. Secondary: 46E35, 26A33.}
\date{\today}
\dedicatory{}
\keywords{Kato square root problem, fractional Laplacian, functional calculus, locally uniform domains, Ahlfors--David regular sets}
\begin{document}
\begin{abstract}

We obtain the Kato square root estimate for second order elliptic operators in divergence form with mixed boundary conditions on an open and possibly unbounded set in $\R^d$ under two simple geometric conditions: The Dirichlet boundary part is Ahlfors--David regular and a quantitative connectivity property in the spirit of locally uniform domains holds near the Neumann boundary part. This improves upon all existing results even in the case of pure Dirichlet or Neumann boundary conditions. We also treat elliptic systems with lower order terms. As a side product we establish new regularity results for the fractional powers of the Laplacian with boundary conditions in our geometric setup.
\end{abstract}
\maketitle
\section{Introduction and main results}
\label{Sec: Introduction}

Let $L$ be a second order elliptic operator in divergence form on an open, possibly unbounded set $O \subseteq \R^d$, $d \geq 2$, with bounded measurable complex coefficients, formally given by
\begin{align}
\label{L}
Lu = - \sum_{i,j=1}^d \partial_i (a_{ij} \partial_ju) - \sum_{i=1}^d \partial_i(b_{i} u) + \sum_{j=1}^d c_{j} \partial_ju + du.
\end{align}
Let $D$ be a closed, possibly empty, subset of the boundary $\bd O$. We complement $L$ with Dirichlet boundary conditions on $D$ and Neumann boundary conditions on $\bd O \setminus D$. More generally, $L$ can be an $(m \times m)$-system in which case $u$ takes its values in $\IC^m$ and each coefficient is valued in $\cL(\IC^m)$.

Let $V \coloneqq \W^{1,2}_D(O)^m$ be the $\W^{1,2}(O)^m$-closure of smooth functions that vanish in a neighborhood of $D$ (Definition~\ref{Def: W12D}). The superscript $m$ indicates that we consider $\IC^m$-valued functions. As usual, we interpret $L$ as the maximal accretive operator in $\L^2(O)^m$ associated with the sesquilinear form $a: V \times V \to \IC$ defined by
\begin{align}
\label{a}
 a(u,v) = \int_O \sum_{i,j=1}^d   a_{ij} \partial_j u \cdot \cl{\partial_i v} + \sum_{i=1}^d  b_{i} u \cdot \cl{\partial_i v} + \sum_{j=1}^d   c_{j} \partial_j u \cdot \cl{v} + d u \cdot \cl{v}  \d x,
\end{align}
which we assume to satisfy for some $\lambda > 0$ the (inhomogeneous) G\r{a}rding inequality
\begin{align}
\label{Garding}
\Re a(u,u) \geq  \lambda (\|u\|_2^2 + \|\nabla u\|_2^2) \qquad \mathrlap{(u \in V).}
\end{align}
Then $L$ is invertible and there is a unique maximal accretive operator $\sqrt{L}$ in $\L^2(O)^m$ that satisfies $(\sqrt{L})^2 = L$. We give a more detailed account in Section~\ref{Sec: Elliptic operator}. More generally, operators like $\sqrt{L}$ arise from functional calculus. We assume familiarity with this concept but for convenience we supply the necessary background for understanding this paper in Appendix~\ref{Sec: Functional calculus for (bi)sectorial operators}. Further (non-)standard notation is explained in a separate paragraph at the end of this introduction.

The Kato problem is to identify the domain of the square root operator as $\dom(\sqrt{L}) = V$ with equivalent norms. If $L$ is self-adjoint, then this essentially follows from the (formal) calculation
\begin{align}
\label{Kato sa calculation}
\|u\|_V^2 \approx a(u,u) = (Lu \SP u)_2 = (\sqrt{L} u \SP (\sqrt{L})^\ast u)_2  = (\sqrt{L} u \SP \sqrt{L^\ast} u)_2 = \|\sqrt{L} u\|_2^2,
\end{align}
where the first step uses \eqref{Garding}, the second step is the definition of $L$ and the final step uses self-adjointness in a crucial way. Indeed, this can be turned into a complete proof, Kato's so-called second representation theorem~\cite[VI.\SS 2.6]{Kato}. No such abstract argument can work when $L$ is not self-adjoint~\cite{Counterexample_McIntosh}. The problem becomes incomparably harder and tied to deep results in harmonic analysis. On $O = \R^d$ it was eventually solved by Auscher--Hofmann--Lacey--McIntosh--Tchamitchian in their 2001 breakthrough paper~\cite{Kato-Square-Root-Proof} and extended by four of them to systems~\cite{Kato-Square-Root-Proof-Systems}. For a historical account and connections to other fields of analysis the reader can refer to the introduction of~\cite{Kato-Square-Root-Proof}.

On general open sets $O$ the problem as posed above is wide open till this day. Applications on sets with ``rough'' geometry come from various fields including, with exemplary references, elliptic and parabolic regularity~\cite{RobertJDE, Bonifacius-Neitzel}, Lions' non-autonomous maximal regularity problem~\cite{Achache-Ouhabaz,Fackler} and boundary value problems~\cite{AAM-ArkMath, AE-mixed}. This motivates the search for minimal geometric assumptions that allow to solve the Kato problem. In this paper we improve on all available results (to be reviewed momentarily) through the following

\begin{theorem}
\label{Thm: main kato result}
Let $O \subseteq \R^d$ be an open set and $D \subseteq \bd O$ a closed subset of the boundary. Suppose that $D$ is Ahlfors--David regular (Definition~\ref{Def: AD}) and that $O$ is locally uniform near $\bd O \setminus D$ (Definition~\ref{Def: locally eps-delta}). Then $\dom(\sqrt{L}) = \W^{1,2}_D(O)^m$ holds with equivalence of norms
\begin{align}
\label{Kato estimate}
\|u\|_2 + \|\nabla u \|_2 \approx \|\sqrt{L} u\|_2 \qquad \mathrlap{(u \in \W^{1,2}_D(O)^m)}
\end{align}
and the implicit constants depend on the coefficients of $L$ only through the coefficient bounds.
\end{theorem}

Here, \emph{coefficient bounds} refers to the lower bound $\lambda$ in \eqref{Garding} and a pointwise upper bound $\Lambda$ for the coefficients. The geometric framework in Theorem~\ref{Thm: main kato result} includes the one proposed by Brewster--Mitrea--Mitrea--Mitrea in their influential paper~\cite{TripleMitrea-Brewster} for treating various aspects of mixed boundary value problems in vast generality. It will be discussed in detail in Section~\ref{Sec: Geometry}. We do not require coordinate charts around $\bd O \setminus D$ in any sense and we do not assume that $O$ satisfies the \emph{interior thickness condition}
\begin{align}
\label{ITC}
\exists c>0 \quad \forall x\in O, \, r \leq 1 \; : \; |\B(x,r) \cap O| \geq c |\B(x,r)|.
\end{align}
Those are two main points compared to all earlier results. 

Indeed, for pure Neumann boundary conditions ($D=\emptyset$) the solution of the Kato problem was only known on Lipschitz domains~\cite{AT2, AKM}. Our assumption reduces to $O$ being an $(\eps,\delta)$-domain (with positive radius if it has infinitely many connected components, see Remark~\ref{Rem: consistency with BHT}). For example, $O$ could be the interior of the von Koch snowflake~\cite[Fig.~3.5]{Triebel-Wavelets}. Pure Dirichlet conditions ($D = \bd O$) were first treated in \cite{AT2} on Lipschitz domains. The Lions problem on mixed boundary conditions was solved in \cite{AKM} on a class of Lipschitz domains if $D$ is a Lipschitz submanifold of $\bd O$. An elaboration on their method of proof in \cite{Darmstadt-KatoMixedBoundary} yielded the solution on bounded interior thick sets with Ahlfors--David regular boundary that are locally Lipschitz regular around $\bd O \setminus D$. 

The proof of Theorem~\ref{Thm: main kato result} divides into three steps. They correspond to the three Sections~\ref{Sec: Laplace} - \ref{Sec: Fattening}. Here, we give an informal overview on the strategy of proof and to fix ideas it will be somewhat more convenient to reverse the order of Section~\ref{Sec: Laplace} and Section~\ref{Sec: AKM}.

Some parts of the paper require that $O$ satisfies \eqref{ITC} nonetheless. We avoid any ambiguity by the following convention. In a context where the underlying set is interior thick, we use bold letters for the relevant objects and write for example $\bO$ instead of $O$.

\subsection*{Step 1: Dirac operator framework}

As many before us, we cast the Kato problem in the abstract first order framework of perturbed Dirac operators that was introduced by Axelsson--Keith--McIntosh in their remarkable elaboration of the original proof of the Kato conjecture \cite{AKM-QuadraticEstimates}. More precisely, we use the refinement in \cite{Laplace-Extrapolation} that will allow us to work on interior thick sets with porous boundary (Definition~\ref{Def: Porous}). The latter holds true under our assumptions (Corollary~\ref{Cor: Boundary porous}) and for the moment let us assume in addition that $\bO$ is interior thick in the sense of~\eqref{ITC}. 

Consequently, we can use the Dirac operator framework as a black box. In Section~\ref{Sec: AKM} we explain what is going on ``behind the scenes'' in more detail, but what basically happens is that harmonic analysis (present due to the non-smooth coefficients) is decoupled from geometry (present due to the rough nature of $\bO$). The harmonic analysis is taken care of by the Axelsson--Keith--McIntosh framework. It then turns out that in order to prove our main theorem, still under the additional assumption \eqref{ITC}, we only need the following \emph{higher regularity result} for the Laplacian with boundary conditions on $\bO$.

\subsection*{Step 2: Higher regularity for the Laplacian}

In light of the growing interest in fractional Laplacians in other fields of analysis this might be of independent interest. The (componentwise) Laplacian $-\Delta_{\bD} + 1$ corresponds to $\boldsymbol{a}_{ij} = \delta_{ij}$, $\boldsymbol{b}_i = 0 = \boldsymbol{c}_j$, $\boldsymbol{d}=1$ in~\eqref{L}. Definitions of fractional Sobolev spaces $\W^{\alpha,2}(\bO)$ and $\W^{\alpha,2}_{\bD}(\bO)$ will be given in Section~\ref{Sec: Laplace}. 

\begin{theorem}
\label{Thm: Laplace}
Let $\bO\subseteq \R^d$ be open and interior thick, let $\bD \subseteq \bd \bO$ be closed and Ahlfors--David regular and assume that $\bO$ is locally uniform near $\bd\bO \setminus \bD$. Then there exists $\eps \in (0,\frac{1}{2})$ such that
\begin{align*}
-\Delta_{\bD} + 1 : \W^{1+s,2}_{\bD}(\bO)^m \to \W^{-1+s,2}_{\bD}(\bO)^m
\end{align*}
is an isomorphism for all $s \in (-\eps,\eps)$. Its fractional power domains in $\L^2(\bO)$ are given by
\begin{align*}
\dom((-\Delta_{\bD} + 1)^{\frac{\alpha}{2}}) 
= \begin{cases}
\W^{\alpha,2}_{\bD}(\bO)^m & \text{if $\alpha \in (\frac{1}{2},1+\eps)$,} \\
\W^{\alpha,2}(\bO)^m & \text{if $\alpha \in (0,\frac{1}{2})$}.
\end{cases}
\end{align*}
\end{theorem}
We give the proof in Section~\ref{Sec: Laplace} below. We remark that the results for $s = 0$ and $\alpha = 1$ are elementary consequences of the Lax--Milgram lemma and the Kato problem for self-adjoint operators, respectively. Hence, we are concerned with a question of extrapolation. Compared to certain forerunners~\cite{AKM, Darmstadt-KatoMixedBoundary, Pryde-MixedBoundary} there are two new ingredients that allow us to relax the required geometric quality of $\bO$: Improved complex interpolation theory for (fractional) Sobolev spaces with boundary conditions, recently developed in \cite{BE}, and Netrusov's spectral synthesis~\cite[Ch.~10]{Adams-Hedberg} replacing more naive measure theoretic considerations in \cite{Darmstadt-KatoMixedBoundary}.

\subsection*{Step 3: Eliminating the interior thickness condition}

At this stage the proof of Theorem~\ref{Thm: main kato result} has been completed under the additional assumption \eqref{ITC}. The final step, carried out in Section~\ref{Sec: Fattening}, consists in eliminating this assumption by an \emph{ad hoc} method. 

First, we observe that if an open set $\bO$ can be written as a countable union of disjoint open sets $\bO = \bigcup_i O_i$ in such a way that the canonical identification
\begin{align}
\L^2(\bO)^m \cong \bigotimes_{i} \L^2(O_i)^m
\end{align}
also gives rise to a decomposition of form domains
\begin{align}
\W^{1,2}_{\bD}(\bO)^m \cong \bigotimes_{i} \W^{1,2}_{D_i}(O_i)^m,
\end{align}
then a divergence form operator $\bL$ on $\bO$ with Dirichlet boundary part $\bD$ inherits a decomposition of its functional calculus as
\begin{align}
f(\bL) \cong \bigotimes_{i} f(L_i),
\end{align}
where $L_i \coloneqq \bL|_{O_i}$ is subject to Dirichlet conditions on $D_i \coloneqq \bD \cap \bd O_i$. Obviously we are somewhat sketchy and some caution is needed to make such decomposition precise, see Section~\ref{Sec: Fattening}. Solving the Kato problem for the triple $(\bL, \bO, \bD)$ will therefore be equivalent to solving it for all triples $(L_i, O_i, D_i)$ with uniform control of the constants in $i$ (Proposition~\ref{Prop: Regularity decomposition for the square root}).

This being said, we reverse the order of reasoning. We let $(L_0, O_0, D_0) \coloneqq (L,O,D)$ be the original operator in Theorem~\ref{Thm: main kato result} and construct an open set $O_1$ disjoint to $O$ with Dirichlet part $D_1 \coloneqq \bd O_1$ such that $\bO \coloneqq O_0 \cup O_1$ with Dirichlet part $\bD \coloneqq D_0 \cup D_1$ is a ``fattened version'' of $(O,D)$: It has the same geometric quality and additionally satisfies \eqref{ITC}. Then we set $L_0 \coloneqq L$ and $L_1 \coloneqq -\Delta_{\bd O_1} +1$. In the interior thick setting we have already solved the Kato problem. Hence, we have the solution for $(\bL, \bO, \bD)$ and obtain the solution for $(L,O,D)$ by restriction to that triple.

\subsection*{Generalizations and consequences}

Once the $\L^2$ comparability \eqref{Kato estimate} has been established, extensions to $\L^p$ norms can be obtained from weak-type estimates and interpolation. On domains this topic has been treated in~\cite{ABHR, E}. To a large extent the geometric setting has been dictated by what was known to be sufficient for the $\L^2$ result. We suggest that all results in \cite{ABHR} and \cite{E} can be extended to the setup of Theorem~\ref{Thm: main kato result} -- at least when $O$ is bounded and connected. 

\subsection*{Acknowledgment}

We would like to thank Patrick Tolksdorf for inspiring discussions on the topic and the anonymous referee for a very thoughtful reading of our manuscript. Sebastian Bechtel would like to thank the ``Studienstiftung des deutschen Volkes'' for financial and academic support. The research of Moritz Egert is partly supported by the ANR project RAGE ANR-18-CE40-0012.

\subsection*{(Non-)Standard notation} 

All vector spaces will be over the complex numbers. Hence, if $H$ is a Hilbert space, then we consider its anti-dual space $H^*$ of bounded conjugate-linear functionals on $H$. By the Riesz isomorphism, $H$ and $H^\ast$ can be identified. The inner product on $H$ is denoted by $(\,\cdot\, |\,\cdot\,)_H$ or simply $(\,\cdot\,|\,\cdot\,)$ if the context is clear. If $T$ is an operator in $H$, we write $\dom(T)$, $\Rg(T)$ and $\Ke(T)$ for its \emph{domain}, \emph{range} and \emph{null space}. Usually, we equip the domain with the graph norm and consider range and null space as subspaces of $H$. We write $H^m$ for the $m$-fold product with $\ell^2$-norm. Likewise if $(H_i)_i$ is a sequence of Hilbert spaces and $T_i$ is an operator in $H_i$ for each $i$, then $\bigotimes_i H_i$ is the Hilbert space of sequences $(U_i)_i$ with $U_i\in H_i$ and $\|(U_i)_i\|\coloneqq (\sum_i \|U_i\|_{H_i}^2)^{1/2} < \infty$ and the operator $\bigotimes_i T_i$ acts componentwise on its domain $\bigotimes_i \dom(T_i) \subseteq \bigotimes_i H_i$. 

Euclidean norms on finite dimensional spaces are denoted by $|\cdot|$. The \emph{Lebesgue measure} in $\R^d$ is also denoted by that symbol. 
We write $\diam(\,\cdot\,)$ and $\dist(\cdot \,,\cdot)$ for the \emph{diameter} and (semi) \emph{distance} of sets in $\R^d$. We use the symbols $\lesssim$ and $\approx$ for one- and two-sided inequalities that hold up to multiplication by a constant that does not depend on the quantified parameters and that depends on the coefficients of $L$ only through the coefficient bounds. 
\section{Discussion of the geometric setup}
\label{Sec: Geometry}

Throughout we work in Euclidean space $\R^d$, $d \geq 2$. We write $B = \B(x,r)$ for the open ball of radius $\rad(B) = r$ centered at $x \in \R^d$ and $cB$ for the concentric ball of radius $c \rad(B)$. We use analogous notation for cubes $Q$ of side length $\len(Q)$. By $\cH^s(E)$, $s \in (0,d]$, we denote the $s$-dimensional \emph{Hausdorff measure} of $E\subseteq \R^d$ defined as follows. For $\varepsilon>0$ we put $$\cH^s_\varepsilon(E)\coloneqq \inf \Bigl\{ \sum_i \rad(B_i)^s: \bigcup_i B_i \supseteq E, \rad(B_i) \leq \varepsilon \Bigr\}$$ and since this value is increasing as $\varepsilon \to 0$ we define $\cH^s(E) \coloneqq \lim_{\varepsilon\to 0} \cH^s_\varepsilon(E)$, see \cite[Sec.~5.1]{Adams-Hedberg}. 

\begin{definition}
\label{Def: AD}
A closed set $D \subseteq \R^d$ is \emph{Ahlfors--David regular} if there is comparability
\begin{align}
\label{eq: l-set}
\cH^{d-1}(B \cap D) \approx \rad(B)^{d-1}
\end{align}
uniformly for all open balls $B$ of radius $\rad(B) < \diam(D)$ centered in $D$.
\end{definition}

\begin{remark}
\label{Rem: AD}
A practical interpretation is that Ahlfors--David regular sets behave $(d-1)$-dimensional at all scales. When $D$ is bounded, then for any $C \in (0,\infty)$ the restriction on the radius may equivalently be changed to $\rad(B) \leq C$, at the cost of changing the implicit constants, see~\cite[Lem.~A.4]{BE}. Sets with the latter property for $C=1$ are called $(d-1)$-sets in~\cite{JW}.
\end{remark}

\subsection{Locally uniform domains}

\begin{definition}
\label{Def: locally eps-delta}
Let $\eps \in (0,1]$ and $\delta \in (0,\infty]$.  Let $O \subseteq \R^d$ be open and $N \subseteq \bd O$. Set $N_\delta  \coloneqq \{z \in \R^d : \dist(z,N) < \delta \}$. Then $O$ is called \emph{locally an $(\eps,\delta)$-domain near $N$} if the following properties hold.

\begin{enumerate}
	\item All points  $x,y \in O \cap N_\delta$ with $|x-y| < \delta$ can be joined in $O$ by an \emph{$\eps$-cigar with respect to $\bd O \cap N_\delta$}, that is to say, a rectifiable curve $\gamma \subseteq O$ of length $\ell(\gamma) \leq \eps^{-1}|x-y|$ such that
	\begin{align}
	\label{eq: eps-delta}
	\dist(z, \bd O \cap N_\delta) \geq \frac{\eps|z-x|\, |z-y|}{|x-y|} \qquad \mathrlap{(z \in \gamma).}
	\end{align}
	
	\item $O$ has \emph{positive radius near $N$}, that is, there exists $c > 0$ such that all connected components $O'$ of $O$ with $\bd O' \cap N \neq \emptyset$ satisfy $\diam(O') \geq c$. 
\end{enumerate}

If the values of $\eps,\delta, c$ need not be specified, then $O$ is simply called \emph{locally uniform near~$N$}.
\end{definition}

\begin{remark}
\label{Rem: locally eps-delta}
Definition~\ref{Def: locally eps-delta} describes a quantitative local connectivity property of $O$ near $N$. For an illustration of $\eps$-cigars with respect to $\bd O$ the reader can refer for instance to~\cite[Fig.~3.1]{Triebel-Wavelets}. Having positive radius is of course only a restriction if $O$ has infinitely many connected components. This condition will be needed only once in our paper, namely in Theorem~\ref{Thm: W12D extension} below.
\end{remark}

Condition \eqref{eq: eps-delta} originates from Jones' influential paper~\cite{Jones_ExtensionOperator}. For his $(\eps,\delta)$-domains he requires that all $x,y \in O$ with $|x-y| < \delta$ can be joined by an $\eps$-cigar with respect to $\bd O$. Locally $(\eps,\delta)$-domains near a part of the boundary have been pioneered in \cite{TripleMitrea-Brewster} using $(\eps,\delta)$-domains as charts around $N$ in analogy with how Lipschitz graphs give rise to the notion of sets with Lipschitz boundary. Our novel definition avoids charts and is inspired by the preprint~\cite{BHT}. Let us give a concise comparison with these two works.

\begin{proposition}
\label{Prop: consistency with BHT}
Let $O \subseteq \R^d$ be an open set and let $N \subseteq \bd O$. 
\begin{enumerate}
	\item If $O$ is locally an $(\eps,\delta)$-domain near $N$ in the sense of \cite{TripleMitrea-Brewster}, then it is locally uniform near $N$ in the sense of Definition~\ref{Def: locally eps-delta}.
	
	\item If $O$ is locally uniform near $N$ in the sense of Definition~\ref{Def: locally eps-delta}, then $(O,N)$ is an admissible geometry in \cite{BHT}.
\end{enumerate}
\end{proposition}

\begin{proof}
We start with (i). Besides further quantitative conditions, being an $(\eps,\delta)$-domain near $N$ in the sense of \cite[Def.~3.4]{TripleMitrea-Brewster} means that there exist (at most) countably many open sets $U_i$ and constants $r,c>0$ such that
\begin{enumerate}
	\item[\quad(a)] for each $x \in N$ there exists some $i$ such that $B(x,8r) \subseteq U_i$ and
	\item[\quad (b)] for each $i$ there is an $(\eps,\delta)$-domain $O_i$ with connected components all of diameter at least $c$ such that $O \cap U_i = O_i \cap U_i$.
\end{enumerate}
We take $\delta' \coloneqq \min(\delta, \eps r)$. Note that in particular $\delta' \leq r$. We show that $O$ is locally an $(\eps,\delta')$-domain near $N$ in the sense of Definition~\ref{Def: locally eps-delta}. To this end, let $x,y \in O \cap N_{\delta'}$ be such that $|x-y|< \delta'$. According to (a) there is a ball $B$ of radius $r$ centered in $N$ and an index $i$ such that $x,y \in 2B$ and $8B \subseteq U_i$. Due to (b) we have $x,y \in O_i$. Consequently, there is a rectifiable curve $\gamma \subseteq O_i$ of length $\ell(\gamma) \leq \eps^{-1}|x-y|$ that joins $x$ to $y$ in such a way that
\begin{align}
\label{eq1: consistency with BHT}
\dist(z, \bd O_i) \geq \frac{\eps|z-x|\, |z-y|}{|x-y|} \qquad \mathrlap{(z \in \gamma).}
\end{align}
From $\ell(\gamma) < r$ we obtain $\gamma \subseteq 3B$ and (b) yields $\gamma \subseteq O_i \cap U_i \subseteq O$. Given $z \in \gamma$, we let $z'$ be a point in $\cl{\bd O \cap N_{\delta'}}$ closest to $z$. We have $|z-z'| \leq 3r$ since $B$ is centered in $N$, which shows that $z' \in \bd O \cap 6B$. Now, (b) yields $z' \in \bd O_i$. Thus we proved $\dist(z, \bd O_i) \leq \dist(z, \bd O \cap N_{\delta'})$ and by \eqref{eq1: consistency with BHT} we see that $\gamma$ is an $\eps$-cigar with respect to $\bd O \cap N_{\delta'}$.

Let $O'$ be a connected component of $O$ with $\bd O' \cap N \neq \emptyset$. We complete the proof of (i) by demonstrating $\diam(O') \geq \min(2r,c)$. Suppose we have $\diam(O') < 2r$. As above, we find a ball $B$ and an index $i$ such that $O' \subseteq 2B$ and $8B \subseteq U_i$. From (b) we obtain that $O$ contains all $x \in O_i$ with $\dist(x,O') < 6r$ and that $O' \subseteq O_i$. In particular, $O'$ is an open and connected subset of $O_i$. Since $O'$ is a maximal connected subset of $O$, we also get that no continuous curve $\gamma \subseteq O_i$ can join points from $O_i \setminus O'$ and $O'$. Hence, $O'$ is a connected component of $O_i$ and $\diam(O') \geq c$ follows.

Let us prove (ii). Besides $O$ having positive radius near $N$, for $(O,N)$ being an admissible geometry in the sense of \cite[Ass.~2.2]{BHT} we need that there exists $\eps', \delta', K > 0$ such that 
\begin{enumerate}
	\item[\quad(c)] all $x,y \in O$ with $|x-y| < \delta'$ can be joined by an $\eps'$-cigar $\gamma$ with respect to $\cl{N}$, \emph{not necessarily contained in $O$}, such that $k(z, O) \coloneqq \inf_O k(z, \cdot) \leq K$ for all $z \in \gamma$, where $k(\cdot\,,\cdot)$ denotes the \emph{hyperbolic distance}
	\begin{align}
	k(x',y') \coloneqq \inf_{\substack{\text{$\gamma' \subseteq \R^d \setminus \cl{N}$ rect.} \\ \text{curve from $x'$ to $y'$}}} \; \int_{\gamma'} \dist(z', N)^{-1} \; |\d z'|.
	\end{align}
\end{enumerate}
Let $\eps, \delta$ be as in Definition~\ref{Def: locally eps-delta}. We check (c) for $\delta' \coloneqq \frac{\delta}{2}$, $\eps' \coloneqq \eps$ and $K \coloneqq 1$. If $x,y \in N_\delta$, then we can use the $\eps$-cigar provided by Definition~\ref{Def: locally eps-delta}, on noting that for $z \in \gamma \subseteq O$ we have $k(z,O) = 0$ and $\dist(z, \bd O \cap N_\delta) \leq \dist(z, N)$. Now, suppose $x \notin N_\delta$. Let $\gamma$ be the straight line segment to $y$ and take any $z \in \gamma$. First,
\begin{align*}
\frac{\eps|z-x| |z-y|}{|x-y|} \leq \eps|x-y| < \frac{\delta}{2} \leq \dist(x, N) - \frac{\delta}{2} \leq \dist(z, N)
\end{align*}
shows that $\gamma$ is an $\eps$-cigar with respect to $\cl{N}$. Second, on taking $\gamma'$ as the segment from $z$ to $x$ in the definition of hyperbolic distance, we find $k(z, O) \leq k(z,x) \leq \ell(\gamma') \frac{2}{\delta} \leq 1$.
\end{proof}

\begin{remark}
\label{Rem: consistency with BHT}
By definition, $(O,\bd O)$ is an admissible geometry in \cite{BHT} if and only if $O$ is an $(\eps,\delta)$-domain with positive radius.  Thus, all introduced notions of locally $(\eps,\delta)$-domains near the full boundary imply that $O$ has positive radius. This observation sharpens \cite[Lem.~3.7]{TripleMitrea-Brewster}.
\end{remark}

Bounded Lipschitz domains are locally uniform, see \cite[Rem.~5.11]{Darmstadt-KatoMixedBoundary} or \cite[Prop.~3.8]{Triebel-Wavelets}. In particular, the local $(\eps,\delta)$-condition near $N$ in the sense of \cite{TripleMitrea-Brewster} already comprises Lipschitz regular sets near $N$: In this case $U_i$ would be a bi-Lipschitz image of the unit cube $Q$ and $O_i$ the image of the lower half of $Q$ under the same transformation. The standard example of a fractal locally uniform domain is the von Koch snowflake~\cite[Fig.~3.5]{Triebel-Wavelets}. 

\subsection{Sobolev extensions}

The Hilbert space $\W^{1,2}(O)$ on an open set $O \subseteq \R^d$ is the collection of all $u \in \L^2(O)$ such that $\nabla u \in \L^2(O)^d$ with norm
\begin{align}
\label{intrinsic Sobolev norm}
 \|u\|_{\W^{1,2}(O)} \coloneqq \Big(\|u\|_{\L^2(O)}^2 + \|\nabla u\|_{\L^2(O)^d}^2 \Big)^{1/2}. 
\end{align}
We introduce the subspace of functions that vanish on a subset of $\cl{O}$ as follows.

\begin{definition}
\label{Def: W12D}
Let $O \subseteq \R^d$ be open and $D \subseteq \cl{O}$ be closed. The Hilbert space $\W^{1,2}_D(O)$ is the  $\W^{1,2}(O)$-closure of the set of test functions
\begin{align}
\label{Def: CDinfty}
 \C^\infty_D(O)\coloneqq \Bigl\{ u|_O: u\in \C^\infty_0(\R^d)\; \text{and} \; \dist(\supp(u),D)>0 \Bigr\}.
\end{align}
\end{definition}

For pure Dirichlet conditions we recover $\W^{1,2}_0(O) = \W^{1,2}_{\bd O}(O)$. If $O$ is an $(\eps,\delta)$-domain with positive radius, then Jones~\cite{Jones_ExtensionOperator} constructed a linear and bounded operator $\cE: \W^{1,2}(O) \to \W^{1,2}(\R^d)$ that does not alter functions on $O$. As such, $\cE$ is called an \emph{extension operator}. Since $\C_0^\infty(\R^d)$ is dense in $\W^{1,2}(\R^d)$, we then obtain $\W^{1,2}(O) = \W^{1,2}_\emptyset(O)$. In general we can extend by zero across $D$ and geometric properties of $O$ are only needed near $N$. This philosophy was followed in \cite[Thm.~3.8]{TripleMitrea-Brewster} by ``gluing together'' the Jones' extension operators on their charts. Very recently, Jones' original construction was altered in \cite{BHT} to produce an extension operator if $(O, N)$ is admissible as in (c) above. In view of Proposition~\ref{Prop: consistency with BHT}.(ii) we may state a special case of the main result of \cite{BHT} in the following

\begin{theorem}
\label{Thm: W12D extension}
Let $O \subseteq \R^d$ be open and $D \subseteq \bd O$ be closed. If $O$ is locally uniform near $\bd O \setminus D$, then there is a bounded linear extension operator $\cE: \W^{1,2}_D(O) \to \W^{1,2}_D(\R^d)$.
\end{theorem}

There are further natural choices for the test function class $\C_D^\infty(O)$ in Definition~\ref{Def: CDinfty} that all lead to the same $\W^{1,2}(O)$-closure, see \cite{BHT} for details.

\subsection{Corkscrew condition and porosity} 

We establish the corkscrew condition in our context. As before, we write $N_\delta$ for the set of points with distance to $N$ less than $\delta$.

\begin{proposition}
\label{Prop: locally eps-delta yields corkscrew}
Suppose that $O \subseteq \R^d$ is open and locally an $(\eps,\delta)$-domain near $N \subseteq \bd O$. Then there exists a constant $\kappa \in (0,1]$ such that:
\begin{align}
\label{eq: Def corkscrew Schlauch}
\forall x \in \cl{N_{\delta/2} \cap O},\, r \leq 1 \quad \exists z \in \B(x,r) \; : \; \B(z, \kappa r) \subseteq O \cap \B(x,r).
\end{align}
\end{proposition}

\begin{proof}
Let  $C \coloneqq \frac{1}{2} \min(\delta,c,1)$. It suffices to obtain some $\kappa$ that works for all radii $r \leq C$ and all $x \in N_{\delta/2} \cap O$. Indeed, for $r\leq 1$ we find $z\in \B(x,Cr) \subseteq \B(x,r)$ with $\B(z, (\kappa C)r) \subseteq O\cap \B(x,r)$, so we only need to use $\kappa C$ instead of $\kappa$. Finally, with a constant \emph{strictly} smaller than $\kappa C$ and a limiting argument, we can allow all $x\in \overline{N_{\delta/2} \cap O}$.

Let $x \in N_{\delta/2} \cap O$. We claim that there is $y \in O$ satisfying $\frac{r}{2} \leq |x-y| \leq \frac{3r}{4}$. Suppose this was not true and let $O'$ be the connected component of $O$ that contains $x$. Then $O' \subseteq \B(x, \frac{r}{2}) \subseteq N_\delta$ and we also have $\B(x,\delta) \cap N_\delta \cap O \subseteq O'$ since all points in the left-hand set can be joined to $x$ via a curve in $O$. The first inclusion gives $\diam(O') < c$, whereas the second one gives $\bd O' \cap N \neq \emptyset$ in contradiction with Definition~\ref{Def: locally eps-delta}. 

We fix any $y \in O$ as above. Then $|x-y| \leq \frac{\delta}{2}$ and in particular $y \in N_\delta \cap O$. Let $\gamma \subseteq O$ be the joining $\eps$-cigar with respect to $\bd O \cap N_\delta$. By continuity we pick $z \in \gamma$ with $|z-x| = \frac{1}{2}|x-y|$ and verify the required properties for $\kappa \coloneqq \frac{\eps}{8}$. First, we have $\B(z, \kappa r) \subseteq \B(x,r)$ by construction. Second, $r \leq \frac{\delta}{2}$ yields  $\dist(z, \R^d \setminus N_\delta) \geq \frac{\delta}{2} - |x-z| \geq \frac{r}{2}$. Third, $|z-y| \geq \frac{1}{2}|x-y|$ and $|x-y| \geq \frac{r}{2}$ plugged into \eqref{eq: eps-delta} give $\dist(z, \bd O \cap N_\delta) \geq \kappa r$. The last two bounds imply $\B(z, \kappa r) \subseteq \R^d \setminus \bd O$. But as $z \in O$ we must have $\B(z, \kappa r) \subseteq O$.
\end{proof}
The property above is closely related to porosity in the following sense.

\begin{definition}
\label{Def: Porous}
A set $E \subseteq \R^d$ is \emph{porous} if there exists $\kappa \in (0,1]$ with the property that:
\begin{align}
\label{eq: Def porosity}
\forall x \in E, \, r \leq 1 \quad \exists z \in \B(x,r) \; : \; \B(z, \kappa r) \subseteq \B(x,r) \setminus E.
\end{align}
\end{definition}

Proposition~\ref{Prop: locally eps-delta yields corkscrew} entails in particular that $N$ is porous. As a non-trivial example let us mention that Ahlfors--David regular sets are porous~\cite[Lem.~A.7]{BE}. This leads to the following important corollary of Proposition~\ref{Prop: locally eps-delta yields corkscrew}.

\begin{corollary}
\label{Cor: Boundary porous}
Under the assumptions of Theorem~\ref{Thm: main kato result} the full boundary $\bd O$ is porous.
\end{corollary}

\begin{proof}
In view of the examples given above, the claim amounts to showing that the union of two porous sets $E_0$, $E_1$ is again porous. To this end, start without loss of generality with a ball $B$ centered in $E_0$ and obtain a ball $B' \subseteq B \setminus E_0$ with comparable radius. Then either $\frac{1}{2}B' \subseteq B \setminus (E_0 \cup E_1)$ or $\frac{1}{2}B' \cap E_1 \neq \emptyset$. In the first case we are done and in the second case porosity of $E_1$, applied with $r= \frac{1}{2} \rad(B')$ and $x$ an intersection point, furnishes a comparably sized ball $B'' \subseteq B' \setminus E_1 \subseteq B \setminus (E_0 \cup E_1)$.
\end{proof}
\section{Definition of the elliptic operator}
\label{Sec: Elliptic operator}

Throughout, we assume that $O \subseteq \R^d$ is open and $D \subseteq \bd O$ is closed. Identifying $\L^2(O)^m$ with its anti-dual space $(\L^2(O)^m)^*$, we have dense embeddings
\begin{align}
\W^{1,2}_D(O)^m \subseteq\L^2(O)^m \subseteq (\W^{1,2}_D(O)^m)^*.
\end{align}
We assume that the coefficients $a_{ij}, b_i, c_j, d: O \to \cL(\IC^m)$ in \eqref{L} are \emph{bounded and measurable}. We group them as $A \coloneqq (a_{ij})_{ij}$, $b \coloneqq (b_i)_i$ and $c \coloneqq (c_j)_j$ in the coefficient matrix
\begin{align}
\label{eq: coefficient matrix}
\begin{bmatrix} d & c \\ b & A \end{bmatrix} : O \to \cL(\IC^m)^{(1+d) \times (1+d)},
\end{align}
and we introduce for a $\IC^m$-valued function $u$ the gradient $\nabla u \coloneqq (\partial_i u)_i$ as a vector in $(\IC^m)^d$. Here, $i$ and $j$ always refer to column and row notation, respectively. With this notation, the sesquilinear form in \eqref{a} can be rewritten as
\begin{align}
\label{eq: a rewritten}
 a: \W^{1,2}_D(O)^m \times \W^{1,2}_D(O)^m \to \IC, \qquad 
a(u,v) =  \int_O \; \begin{bmatrix} d & c \\ b & A \end{bmatrix} \begin{bmatrix} u \\ \nabla u \end{bmatrix}  \cdot \cl{\begin{bmatrix} v \\ \nabla v \end{bmatrix}} \d x.
\end{align}
Our \emph{ellipticity} assumption is the lower bound \eqref{Garding}. Note also that $(\|\cdot\|_2^2 + \|\nabla \cdot\|_2^2)^{1/2}$ is the Hilbert space norm on $\W_D^{1,2}(O)^m$.

The Lax-Milgram lemma associates with $a$ the bounded and invertible operator
\begin{align*}
\cL: \W^{1,2}_D(O)^m \to (\W^{1,2}_D(O)^m)^*, \qquad \langle \cL u, v \rangle =  a(u,v).
\end{align*}
We define $L$ to be the maximal restriction of $\cL$ to an operator in $\L^2(O)^m$. Then $L$ is an invertible, maximal accretive, sectorial operator in $\L^2(O)$ of some angle $\omega \in [0,\pi/2)$, see~\cite[Prop.~7.3.4]{Haase}. The adjoint $L^*$ is associated in the same way with $a^*(u,v) \coloneqq \cl{a(v,u)}$, see \cite[Thm.~VI\SS 2.5]{Kato}. Consequently, $L$ is self-adjoint when the matrix in \eqref{eq: coefficient matrix} is Hermitian. 

The square root $\sqrt{L}$ is defined via the sectorial functional calculus for $L$. It is invertible and maximal accretive~\cite[Cor.~7.1.13]{Haase}. In particular, it coincides with Kato's original definition of the square root~\cite[Thm.~V\SS 3.35]{Kato} as the unique maximal accretive operator in $\L^2(O)^m$ that satisfies $(\sqrt{L})^2 = L$.

It will be convenient to write $L$ as a composition of differential and multiplication operators, more in the spirit of the formal definition \eqref{L}. To this end, we introduce the closed and densely defined operator 
\begin{align}
\label{eq: nablaD}
\nabla_D: \W^{1,2}_D(O)^m \subseteq\L^2(O)^m \to \L^2(O)^{dm}, \qquad \nabla_D u = \nabla u
\end{align}
and let $-\div_D$ be its (unbounded) adjoint. In view of \eqref{eq: a rewritten} it follows that
\begin{align}
\label{eq: L as composition}
L 
= \begin{bmatrix} 1 & -\div_D \end{bmatrix}
\begin{bmatrix} d & c \\ b & A \end{bmatrix}
\begin{bmatrix}  1 \\ \nabla_D \end{bmatrix}
\end{align}
with maximal domain in $\L^2(O)^m$. Integration by parts reveals $\div_D (u_i)_i = \sum_{i=1}^d \partial_i u_i$ for $(u_i)_i \in (\C_0^\infty(O)^m)^d$ but in this generality no explicit description of $\dom(\div_D)$ is available.
\section{Higher regularity for fractional powers of the Laplacian}
\label{Sec: Laplace}

The goal of this section is to show Theorem~\ref{Thm: Laplace}, thereby accomplishing Step~2 from the introduction. In the whole section we fix $\bO$ and $\bD$ satisfying the assumptions from Theorem~\ref{Thm: Laplace}. It suffices to treat the case $m=1$ since $-\Delta_{\bD}$ in $\L^2(\bO)^m$ acts componentwise. We adopt the convention that function spaces without reference to an underlying set are understood on the whole space, e.g.\ we write $\W^{1,2}$ instead of $\W^{1,2}(\R^d)$. 

\subsection{Fractional Sobolev spaces on open sets with vanishing trace condition}
\label{Sec: Spaces}

We introduce the \emph{(fractional) Sobolev spaces} of regularity $s\in (0,\frac{3}{2})$. Whereas this is a classical topic on $\R^d$, different definitions suiting different purposes are possible on $\bO$. We follow the treatment in~\cite{BE} with a focus on interpolation theory.

If $s \in (0,1)$, then $\W^{s,2}$ consists of all $u\in \L^2$ such that
\begin{align}
	\|u\|_{\W^{s,2}}^2 \coloneqq \|u\|_{\L^2}^2 + \iint_{\substack{x,y \in \R^d \\ |x-y| < 1}} \frac{|u(x) - u(y)|^2}{|x-y|^{d + 2s}} \d x \d y < \infty.
\end{align}
If $s\in (0,\frac{1}{2})$, then $\W^{1+s,2}$ consists of all $u\in \L^{2}$ with $\|u\|_{\W^{1+s,2}}^2 \coloneqq \|u\|_{\L^2}^2 + \|\nabla u\|_{\W^{s,2}}^2 < \infty$. It will be convenient to set $\W^{0,2} \coloneqq \L^2$. It follows from the structure of the norms that these spaces are Hilbert spaces and that $\nabla$ maps $\W^{1+s,2}$ into $\W^{s,2}$. The \emph{Bessel potential space} 
\begin{align}
\label{eq: Bessel spaces}
\H^{s,2} \coloneqq \big\{u \in \L^2 : \|u\|_{\H^{s,2}} \coloneqq \|(1-\Delta)^{\frac{s}{2}}u\|_{\L^2} < \infty \big\}
\end{align}
coincides with $\W^{s,2}$ by \cite[Sec.~2.4.2.~Rem.~2]{Triebel}. Here, $\Delta$ is the Laplacian in~$\R^d$. 

Since $\bD$ is a $(d-1)$-set, see Remark~\ref{Rem: AD} for this terminology, a version of the Lebesgue differentiation theorem allows us to define traces on $\bD$. The following is a weakened version of \cite[Thm.~VI.1]{JW} that suffices for our purpose.

\begin{proposition}
\label{prop: JW}
Let $s \in (\frac{1}{2}, \frac{3}{2})$ and $u \in \W^{s,2}$. For $\cH^{d-1}$-almost every $x \in \bD$ the limit
\begin{align*}
(\Res_{\bD} u)(x) \coloneqq \lim_{r\to 0} \frac{1}{|\B(x,r)|} \int_{\B(x,r)} u(y) \d y
\end{align*}
exists. The restriction operator $\Res_{\bD}$ maps $\W^{s,2}$ boundedly into $\L^2(\bD, \cH^{d-1})$. 
\end{proposition}

With the trace operator at hand, we introduce the closed subspace $\W^{s,2}_{\bD}$ of $\W^{s,2}$ by 
\begin{align}
\W^{s,2}_{\bD}\coloneqq \bigl\{u\in \W^{s,2}: \Res_{\bD} u = 0 \bigr\}.
\end{align}
In the case $s=1$ this notion is consistent with Definition~\ref{Def: W12D}, see for instance~\cite[Lem.~3.3]{BE}. Finally, we denote the distributional restriction to $\bO$ by $|_{\bO}$ and define fractional Sobolev spaces on $\bO$ by restriction.

\begin{definition}
\label{Def: spaces on open sets}
Let $s \in [0,\frac{3}{2})$ and $t \in (\frac{1}{2}, \frac{3}{2})$. Put $\W^{s,2}(\bO) \coloneqq \{ u|_{\bO}: u\in \W^{s,2} \}$ and $\W^{t,2}_{\bD}(\bO) \coloneqq \{ u|_{\bO}: u\in \W^{t,2}_{\bD} \}$ and equip them with quotient norms.
\end{definition}

\begin{remark}
\label{Rem: spaces on open sets}
These spaces are again Hilbert spaces by construction as quotients of Hilbert spaces. Since $\bO$ is interior thick, we have that $\W^{t,2}_{\bD}(\bO)$ is a closed subspace of $\W^{t,2}(\bO)$ with an equivalent norm~\cite[Lem.~3.4]{BE}. As a cautionary tale, let us stress that in the context of this section $\W^{1,2}(\bO)$ is embedded into but possibly \emph{not} equal to the collection of all $u \in \L^2(\bO)$ with $\nabla u \in \L^2(\bO)^d$ and norm \eqref{intrinsic Sobolev norm}. However, as a consequence of Theorem~\ref{Thm: W12D extension}, the definition of $\W^{1,2}_{\bD}(\bO)$ above coincides with the original one from Definition~\ref{Def: W12D} up to equivalent norms.
\end{remark}

Spaces of negative smoothness are defined by duality extending the inner product on~$\L^2(\bO)$.

\begin{definition}
For $s \in [0,\frac{3}{2})$ and $t \in (\frac{1}{2}, \frac{3}{2})$ let $\W^{-s,2}(\bO)$ and $\W^{-t,2}_{\bD}(\bO)$ be the anti-dual spaces of $\W^{s,2}(\bO)$ and $\W^{t,2}_{\bD}(\bO)$, respectively.
\end{definition}

We turn our focus to the density of test functions in these spaces. The following proposition follows as a special case of Netrusov's Theorem \cite[Thm.~10.1.1]{Adams-Hedberg} if one replaces the appearing capacities by the Hausdorff measure using \cite[Thm.~5.1.9]{Adams-Hedberg}. To make the statement more concise, we use the concept of \emph{Hausdorff co-dimension}, defined by
$$\codim_\cH(E)\coloneqq d-\dim_\cH(E),$$
where
$$\dim_\cH(E)\coloneqq \inf \bigl\{ s \in (0,d]: \cH^s(E)=0 \}$$
is the \emph{Hausdorff dimension} of $E \subseteq \R^d$. The measure $\cH^s$ was defined in Section~\ref{Sec: Geometry}.

\begin{proposition}[A Version of Netrusov's Theorem]
\label{Prop: Netrusov}
Let $0<s<\frac{d}{2}$ and let $E \subseteq \R^d$ be closed. If $2s < \codim_\cH(E)$, then $\C^\infty_E$ is dense in $\W^{s,2}$.
\end{proposition}

In order to show that this version of Netrusov's Theorem is applicable in our setting, we need an elementary covering lemma for porous sets \cite[Lemma A.8]{BE}.

\begin{lemma}
	\label{Lem: porous covering property}
	If $E\subseteq \R^d$ is porous, then there exists $C\geq 1$ and $0<t<d$ such that, given $x\in E$ and $0<r<R\leq 1$, there is a covering of $E\cap B(x,R)$ by $C(R/r)^t$ balls of radius $r$ centered in $E$.
\end{lemma}

\begin{proposition}
\label{Prop: porous implies density of test functions and zero extension}
	There exists $0<s_0\leq \frac{1}{2}$ such that $\C^\infty_{\bd \bO}(\bO)$ is dense in $\W^{s,2}(\bO)$ and the zero extension operator $\Ext_0: \W^{s,2}(\bO) \to \W^{s,2}$ is bounded for $0<s<s_0$.
\end{proposition}
\begin{proof}
	Since $\bd \bO$ is porous by Corollary~\ref{Cor: Boundary porous}, we can pick $C\geq 1$ and $0<t<d$ as in Lemma~\ref{Lem: porous covering property}. We take $0<s<1$ such that $2s<d-t$. Then, let $B$ be some ball centered in $\bd \bO$ with $\rad(B)= 1$ and let $0<r<1$. Given $(B_i)_{i}$ a covering of $\bd \bO\cap B$ provided by the lemma and $\ell\in (t,d)$, we estimate
	\begin{align}
			\cH^{\ell}_r(\bd \bO \cap B) \leq \sum_i \rad(B_i)^{\ell} = \#_i \, r^{\ell} \leq C r^{\ell-t}.
	\end{align}
	Taking the limit as $r\to 0$, we arrive at $\cH^{\ell}(\bd \bO \cap B) = 0$. Finally, a countable covering of $\bd \bO$ by such balls yields $\cH^{\ell}(\bd \bO) = 0$, so by definition we have $\dim_\cH(\bd \bO) \leq t$ and therefore $2s<\codim_\cH(\bd \bO)$. Now, Netrusov's Theorem gives density of $\C^\infty_{\bd \bO}$ in $\W^{s,2}$ and the first claim follows by restriction to $\bO$.
	
	Boundedness of the zero extension operator for $s$ sufficiently small is a result due to Sickel~\cite{Sickel}, which is put into the context of sets with porous boundary in \cite[Ex.~2.16 \& Prop.~2.17]{BE}. Inspecting the proof of~\cite[Prop.~A.10]{BE} reveals that the same range of $s$ as before would work but we do not need such precision.
\end{proof}

\subsection{Interpolation scales}
\label{Sec: Interpolation scales}

We require standard notions from interpolation theory of Hilbert spaces and refer to the monograph~\cite{Triebel} for background. We denote the $(\theta,2)$-real and $\theta$-complex interpolation brackets by $(\,\cdot\,,\cdot\,)_{\theta,2}$ and $[\,\cdot\,,\cdot\,]_\theta$, respectively.

\begin{definition}
\label{Def: Interpolation scale}
Let $I \subseteq \R$ be an interval and for each $i \in I$ let $H_i$ be a Hilbert space. Call $(H_i)_i$ a \emph{complex interpolation scale} if whenever $i_0 < i_1$, then $H_{i_1} \subseteq H_{i_0}$ with dense and continuous inclusion and
\begin{align}
\label{eq: Interpolation scale}
[H_{i_0}, H_{i_1}]_\theta = H_{(1-\theta)i_0 + \theta i_1} 
\end{align}
for all $\theta \in (0,1)$ up to equivalent norms.
\end{definition}

We could have introduced a similar notion for $(\theta,2)$-real interpolation, but it is also a good opportunity to note that this coincides with $\theta$-complex interpolation when working with Hilbert spaces \cite[Cor. C.4.2.]{Hytonen-Vol-I}. We will freely use this fact. If $H_0, H_1$ are Hilbert spaces with dense inclusion $H_1 \subseteq H_0$, then $([H_0,H_1]_\theta)_{\theta \in [0,1]}$ is a complex interpolation scale, see \cite[Sec.~1.9.3.~Thm.~(c) \& (d)]{Triebel} and \cite[Sec.~1.10.3.~Thm.~2]{Triebel}. In this context \eqref{eq: Interpolation scale} is called \emph{reiteration} in the literature and we use the convention $[H_0, H_1]_j = H_j$ for $j = 0,1$.

For the proof of Theorem~\ref{Thm: Laplace} we need the following interpolation scales:
\begin{alignat}{2}
&\mathrm{(a)} \quad (\W^{s,2})_{s \in [0, \frac{3}{2})}, \hspace{50pt}
&&\mathrm{(b)} \quad (\W^{s,2}(\bO))_{s \in (0, \frac{3}{2})}, \\
&\mathrm{(c)} \quad (\W_{\bD}^{s,2}(\bO))_{s \in (\frac{1}{2}, \frac{3}{2})}, \hspace{50pt}
&& \mathrm{(d)} \quad (\W_{\bD}^{s,2}(\bO))_{s \in (-\frac{3}{2}, -\frac{1}{2})}, \\
&\mathrm{(e)} \quad (\dom((1-\Delta_{\bD})^{\frac{s}{2}}))_{s \in [0,\infty)}.
\end{alignat}
Part (a) is classical~\cite[Sec.~2.4.4.~Thm.~1]{Triebel} and since $1-\Delta_{\bD}$ is self-adjoint, so is (e), see~\cite[Thm.~6.6.9 \& Cor.~7.1.6]{Haase}. Also, (d) follows from (c) by duality~\cite[Sec.~1.11.2]{Triebel}. Parts (b) and (c) use the standing geometric assumptions and have been obtained in \cite[Prop.~3.8]{BE} and \cite[Thm.~1.2]{BE}, respectively. We need a minor but yet important modification of (b) that uses again porosity of $\bd \bO$.

\begin{lemma}
\label{Lem: Die ganze Scheisse}
In the complex interpolation scale (b) one can add $s=0$.
\end{lemma}

\begin{proof}
This follows from two abstract interpolation principles. First, with $s_0$ as in Proposition~\ref{Prop: porous implies density of test functions and zero extension}, $|_{\bO}$ the pointwise restriction to $\bO$ and $\Ext_0$ the zero extension operator, we get that
\begin{align}
\Ext_0: \W^{s,2}(\bO) \to \W^{s,2}, \qquad |_{\bO}: \W^{s,2} \to \W^{s,2}(\bO) \qquad \mathrlap{(0 \leq s < s_0)}
\end{align}
are bounded operators with $|_{\bO} \,\circ\, \Ext_0$ the identity on $\W^{0,2}(\bO)$. By \cite[Sec.~1.2.4]{Triebel} the property of being a complex interpolation scale inherits from  $(\W^{s,2})_{s \in [0,s_0)}$ to $(\W^{s,2}(\bO))_{s \in [0,s_0)}$. Second, the latter scale overlaps with the one in (b) and Wolff's theorem~\cite[Thm.~1]{Wolff-refined} states that they can be ``glued together'' to an interpolation scale on $[0,\frac{3}{2})$.
\end{proof}

\subsection{Mapping properties for $1-\Delta_{\textbf{\textit{D}}}$}
\label{Sec: Mapping properties for Laplacian}

We start with mapping properties for the distributional gradient on spaces of fractional smoothness.

\begin{lemma}
\label{Lem: mapping property gradient}
Let $s, t > 0$ satisfy $t< \frac{1}{2}$ and $s<s_0$, where $s_0$ was determined in Proposition~\ref{Prop: porous implies density of test functions and zero extension}. Then $\nabla$ is a bounded operator $\W^{1+t,2}(\bO)\to \W^{t,2}(\bO)^d$ and $\W^{1-s,2}(\bO)\to \W^{-s,2}(\bO)^d$.
\end{lemma}

\begin{proof}
For the first claim we simply note that $\nabla: \W^{1+t,2} \to (\W^{t,2})^d$ is bounded and that $\nabla Eu \in (\W^{t,2})^d$ is an extension of $\nabla u$ if $Eu \in \W^{1+t,2}$ is an extension of $u \in \W^{1+t,2}(\bO)$. 

For the second claim let $u\in \W^{1-s,2}(\bO)$ and let $Eu\in \W^{1-s,2}$ be some extension of $u$. Let $i=1,\ldots,d$. Given $\varphi\in \C^\infty_{\bd \bO}(\bO)$, we first rewrite the duality pairing as
\begin{align}
 -\langle \partial_i  u, \varphi \rangle
 &=(u\,|\, \partial_i \varphi)_{\L^2(\bO)} \\
 &= (Eu\,|\,\Ext_0 \partial_i\varphi)_{\L^2} \\
 &= (Eu\,|\,\partial_i\Ext_0\varphi)_{\L^2}\\
 &= \big((1-\Delta)^{\frac{1-s}{2}} Eu \,\big| \, \partial_i (1-\Delta)^{-\frac{1}{2}} (1-\Delta)^{\frac{s}{2}} \Ext_0 \varphi \big)_{\L^2},
\end{align}
where we used again the fractional powers of the Laplacian on $\R^d$ and commuted the respective Fourier multiplication operators. Since the Riesz transform $\partial_i (1-\Delta)^{-\frac{1}{2}}$ is bounded on $\L^2$ by Plancherel's theorem and the Bessel spaces in \eqref{eq: Bessel spaces} coincide with the fractional Sobolev spaces, we obtain
\begin{align}
 |\langle \partial_i  u, \varphi \rangle| &\lesssim \|Eu\|_{\W^{1-s}} \|\Ext_0\varphi\|_{\W^{s,2}} \\
 &\lesssim \|Eu\|_{\W^{1-s}} \|\varphi\|_{\W^{s,2}(\bO)},
\end{align}
where the final step is due to Proposition~\ref{Prop: porous implies density of test functions and zero extension}. By passing to the infimum over all extensions $Eu$ we arrive at $|\langle \partial_i u, \varphi \rangle| \lesssim \|u\|_{\W^{1-s,2}(\bO)} \|\varphi\|_{\W^{s,2}(\bO)}$. But since $\C^\infty_{\bd \bO}(\bO)$ is dense in $\W^{s,2}(\bO)$ by Proposition~\ref{Prop: porous implies density of test functions and zero extension}, this shows $\partial_i u \in \W^{-s,2}(\bO)$ with $\|\partial_i u\|_{\W^{-s,2}(\bO)} \lesssim \|u\|_{\W^{1-s,2}(\bO)}$.
\end{proof}

\begin{proposition}
\label{Prop: mapping properties Laplacian}
Let $0<s<s_0$ with $s_0$ as in Proposition~\ref{Prop: porous implies density of test functions and zero extension}. Then $1-\Delta_{\bD}: \W^{1,2}_{\bD}(\bO) \to \W^{-1,2}_{\bD}(\bO)$ restricts/extends to a bounded operator $\W^{1\pm s,2}_{\bD}(\bO) \to \W^{-1\pm s,2}_{\bD}(\bO)$.
\end{proposition}

\begin{proof}
By definition, we have
\begin{align*}
\langle (1-\Delta_{\bD})u, v \rangle 
= (u \mid v)_{2} + (\nabla u \mid \nabla v)_{2} \qquad \mathrlap{(u,v \in \W^{1,2}_{\bD}(\bO)).}
\end{align*}
If in addition $u \in \W^{1\pm s,2}_{\bD}(\bO)$ and $v \in \W^{1\mp s,2}_{\bD}(\bO)$, then we control the first part of the right-hand side by
\begin{align*}
|(u \mid v)_{2}| \leq \|u\|_2 \|v\|_2 \leq \|u\|_{\W^{1\pm s,2}_{\bD}(\bO)} \|v\|_{\W^{1\mp s,2}_{\bD}(\bO)},
\end{align*}
whereas for the second part we first use that the $\W^{\pm s,2}(\bO)$ - $\W^{\mp s,2}(\bO)$ duality extends the inner product on $\L^2(\bO)$ and then apply Lemma~\ref{Lem: mapping property gradient} to give
\begin{align*}
| (\nabla u \mid \nabla v)_{2}| \leq \|\nabla u\|_{\W^{\pm s,2}(\bO)} \|\nabla v\|_{\W^{\mp s,2}(\bO)} \lesssim \|u\|_{\W^{1\pm s,2}_{\bD}(\bO)} \|v\|_{\W^{1\mp s,2}_{\bD}(\bO)}.
\end{align*}
By virtue of the scale (c) from Section~\ref{Sec: Interpolation scales} the fractional Sobolev spaces form a hierarchy of densely included spaces. This yields the claim.
\end{proof}

We continue by recalling an abstract extrapolation result due to {\u{S}}ne{\u{\ii}}berg \cite{Sneiberg-Original,ABES3}.  

\begin{proposition}
	\label{Prop: Sneiberg}
	Let $a< b$, let $(H_i)_{i \in [a,b]}$ and $(K_i)_{i \in [a,b]}$ be complex interpolation scales and let $T: H_a \to K_a$ be a bounded linear operator that is also $H_b \to K_b$ bounded. Then the set $\{i \in (a,b) \mid T: H_i \to K_i \text{ is an isomorphism} \}$ is open.
\end{proposition}

This enables us to complete the first part of Theorem~\ref{Thm: Laplace} through the following

\begin{proposition}
\label{Prop: Laplacian isomorphism extrapolates}
There exists $\varepsilon \in (0, \frac{1}{2})$ such that $1-\Delta_{\bD}$ is an isomorphism between $\W^{1+s,2}_{\bD}(\bO)$ and $\W^{-1+s,2}_{\bD}(\bO)$ for all $s \in (-\eps,\eps)$.
\end{proposition}

\begin{proof}
We define $\beta\coloneqq \frac{1}{2} s_0$ and $I = [-\beta,\beta]$. According to Proposition~\ref{Prop: mapping properties Laplacian}, $1-\Delta_{\bD}$ extends to a bounded operator $\W^{1-\beta,2}_{\bD}(\bO) \to \W^{-1-\beta,2}_{\bD}(\bO)$. We denote this extension by $T$ and the same proposition shows that $T$ also restricts to a bounded operator $\W^{1+\beta,2}_{\bD}(\bO) \to \W^{-1+\beta,2}_{\bD}(\bO)$. From Section~\ref{Sec: Interpolation scales} we know that $(\W^{1+s,2}_{\bD}(\bO))_{i \in I}$ and $(\W^{-1+s,2}_{\bD}(\bO))_{i \in I}$ are complex interpolation scales. We know by the Lax--Milgram lemma that $s=0$ is contained in the isomorphism set in {\u{S}}ne{\u{\ii}}berg's Theorem, thus the claim follows.
\end{proof}

\subsection{Domains of fractional powers of the Laplacian}
\label{Sec: Laplace fractional power domains}

In this section we complete the proof of Theorem~\ref{Thm: Laplace}. 

Starting point for our investigation is the following inclusion from \cite[Lem.~6.10]{BE}.

\begin{lemma}
\label{Lem: Schwierige Inklusion}
Let $s \in (0,1)$, $s \neq \frac{1}{2}$, and define $\vartheta$ via $\vartheta(s+\frac{1}{2}) = 1$. Then the following set inclusion holds:
\begin{align}
\label{eq: Schwierige Inklusion}
[\L^2(\bO), \W^{s+\frac{1}{2},2}_{\bD}(\bO)]_{\vartheta s} \supseteq \begin{cases}
\W_{\bD}^{s,2}(\bO) &(\text{if $s > \frac{1}{2}$}) \\
\W^{s,2}(\bO) &(\text{if $s < \frac{1}{2}$})
\end{cases}.
\end{align}
\end{lemma}

We show that the converse inclusion to \eqref{eq: Schwierige Inklusion} is continuous. This was done in \cite{BE} for sets with Ahlfors--David regular boundary but the following simplified argument for porous boundary would have already worked there.

\begin{lemma}
\label{Lem: Einfache Inklusion}
Let $s \in (0,1)$, $s \neq \frac{1}{2}$. The converse inclusion to \eqref{eq: Schwierige Inklusion} is continuous.
\end{lemma}

\begin{proof}
Since $\W^{s+1/2,2}_{\bD}(\bO) \subseteq \W^{s+1/2,2}(\bO)$, we can use the interpolation scale (b) in Section~\ref{Sec: Interpolation scales} (taking Lemma~\ref{Lem: Die ganze Scheisse} into account), to give
\begin{align}
\label{eq1: Einfache Inklusion}
[\L^2(\bO), \W^{s+\frac{1}{2},2}_{\bD}(\bO)]_{\vartheta s} \subseteq [\L^2(\bO), \W^{s+\frac{1}{2},2}(\bO)]_{\vartheta s} = \W^{s,2}(\bO)
\end{align}
with continuous inclusion. In the case $s<\frac{1}{2}$ this completes the proof. In the case $s>\frac{1}{2}$ we have to recover the Dirichlet condition on the right-hand side. Call $X$ the interpolation space on the left-hand side. Then $\W^{s+1/2,2}_{\bD}(\bO)$ is dense in $X$ by definition of an interpolation scale. By \eqref{eq1: Einfache Inklusion} we find that $X$ is contained in the closure of $\W^{s+1/2,2}_{\bD}(\bO)$ for the $\W^{s,2}(\bO)$-norm. But $\W^{s+1/2,2}_{\bD}(\bO) \subseteq \W^{s,2}_{\bD}(\bO)$ and the latter is a closed subspace of $\W^{s,2}(\bO)$, see Remark~\ref{Rem: spaces on open sets}. Thus, we have $X \subseteq \W^{s,2}_{\bD}(\bO)$ with continuous inclusion.
\end{proof}

By the bounded inverse theorem \eqref{eq: Schwierige Inklusion} now becomes an equality with equivalent norms. An application of the reiteration property of interpolation spaces (see Section~\ref{Sec: Interpolation scales}) reveals the interpolation spaces between $\L^2(\bO)$ and $\W^{1,2}_{\bD}(\bO)$.

\begin{proposition}
\label{Prop: Interpolation}
For $s\in (0,1)$ one has the following interpolation identity
\begin{align}
[\L^2(\bO), \W^{1,2}_{\bD}(\bO)]_s = \begin{cases}
\W^{s,2}(\bO) &(\text{if $s < 1/2$}) \\
\W_{\bD}^{s,2}(\bO) &(\text{if $s > 1/2$})
\end{cases}.
\end{align}
\end{proposition}

\begin{proof}
All equalities in this proof will be up to equivalent norms. We take $0<t<\frac{1}{2}<s<1$. First, we determine the interpolation space for the parameter $s$. Observe that $\vartheta$ in Lemma~\ref{Lem: Schwierige Inklusion} is given by $\vartheta = \frac{2}{2s+1}$ and consequently $\frac{1}{\vartheta} = s + \frac{1}{2}$. From $\vartheta s < \vartheta < 1$ follows the existence of some $\lambda \in (0,1)$ such that $(1-\lambda)\vartheta s + \lambda = \vartheta$. By reiteration, the preparatory interpolation identity from above and the scale (c) in Section~\ref{Sec: Interpolation scales} we arrive at
\begin{align}
[\L^2(\bO), \W^{s+\frac{1}{2},2}_{\bD}(\bO)]_\vartheta 
&= \Bigl[[\L^2(\bO), \W^{s+\frac{1}{2},2}_{\bD}(\bO)]_{\vartheta s}, \W^{s+\frac{1}{2},2}_{\bD}(\bO)\Bigr]_\lambda \\
&= [\W^{s,2}_{\bD}(\bO), \W^{s+\frac{1}{2},2}_{\bD}(\bO)]_\lambda \\
&= \W^{1,2}_{\bD}(\bO).
\end{align}
We readily deduce by the reiteration property
\begin{align*}
[\L^2(\bO), \W^{1,2}_{\bD}(\bO)]_s 
&= [\L^2(\bO), [\L^2(\bO), \W^{s+\frac{1}{2},2}_{\bD}(\bO)]_\vartheta]_s 
= [\L^2(\bO), \W^{s+\frac{1}{2},2}_{\bD}(\bO)]_{\vartheta s} \\
&= \W^{s,2}_{\bD}(\bO).
\end{align*}
Likewise, we put $\vartheta\coloneqq \frac{2}{2t+1}$. Then we obtain from the reiteration property, the previous case with $s=t+\frac{1}{2}$ and the preparatory interpolation identity that
\begin{align}
[\L^2(\bO), \W^{1,2}_{\bD}(\bO)]_t 
&= \Bigl[\L^2(\bO), [\L^2(\bO), \W^{1,2}_{\bD}(\bO)]_{t+\frac{1}{2}}\Bigr]_{\vartheta t} 
= [\L^2(\bO), \W^{t+\frac{1}{2},2}_{\bD}(\bO)]_{\vartheta t} \\
&= \W^{t,2}(\bO). \qedhere
\end{align}
\end{proof}

\begin{proof}[Proof of Theorem~\ref{Thm: Laplace}]
In view of Proposition~\ref{Prop: Laplacian isomorphism extrapolates} it only remains to determine the fractional power domains in $\L^2(\bO)$. The starting point is that $1-\Delta_{\bD}$ is a self-adjoint operator and therefore we have $\dom((1-\Delta_{\bD})^\frac{1}{2}) = \W^{1,2}_{\bD}(\bO)$ by the Kato property for self-adjoint operators. Combining Proposition~\ref{Prop: Interpolation} and the interpolation scale (e) in Section~\ref{Sec: Interpolation scales} we obtain 
\begin{align}
\dom((1-\Delta_{\bD})^\frac{\alpha}{2})= \begin{cases}
\W_{\bD}^{\alpha,2}(\bO) &(\text{if $\alpha > 1/2$}) \\
\W^{\alpha,2}(\bO) &(\text{if $\alpha < 1/2$})
\end{cases}
\end{align}
for $\alpha \in [0,1]$, all with equivalent norms. For the extrapolation we decompose fractional powers of $1-\Delta_{\bD}$ above $\frac{1}{2}$ into the inverse of the full operator and fractional powers of lower order. Since the fractional powers of $1-\Delta_{\bD}$ are invertible, see Example~\ref{Ex: Fractional powers},
$$(1-\Delta_{\bD})^\frac{1-\alpha}{2}: \W^{1-\alpha,2}_{\bD}(\bO)\to \L^2(\bO)$$
is an isomorphism for $\alpha \in [0, \frac{1}{2})$. Using duality and self-adjointness, we moreover get that $(1-\Delta_{\bD})^\frac{1-\alpha}{2}$ extends to an isomorphism between $\L^2(\bO)$ and $\W^{-1+\alpha,2}_{\bD}(\bO)$. In particular, with $\varepsilon$ from Proposition~\ref{Prop: Laplacian isomorphism extrapolates} and $\alpha \in (0,\varepsilon)$, $(1-\Delta_{\bD})^\frac{1-\alpha}{2}$ maps into the domain of the extrapolated Lax--Milgram isomorphism. On the dense subset $\dom((1-\Delta_{\bD})^\frac{1-\alpha}{2})$ of $\L^2(\bO)$ we have the decomposition $$(1-\Delta_{\bD})^{-\frac{1+\alpha}{2}}=(1-\Delta_{\bD})^{-1}(1-\Delta_{\bD})^\frac{1-\alpha}{2}.$$
Again by example~\ref{Ex: Fractional powers}, the left-hand side is an isomorphism from $\L^2(\bO)$ onto $\dom((1-\Delta_{\bD})^{\frac{1+\alpha}{2}})$. But the right-hand side extends to an isomorphism onto $\W^{1+\alpha,2}_{\bD}(\bO)$, which reveals that indeed $\dom((1-\Delta_{\bD})^\frac{1+\alpha}{2})=\W^{1+\alpha,2}_{\bD}(\bO)$.
\end{proof}
\section{Proof of Theorem~\ref{Thm: main kato result} on interior thick sets}
\label{Sec: AKM}

This section corresponds to Step~1 of the introduction. Throughout we assume that $\bO\subseteq \R^d$ is an open and interior thick set, that $\bD \subseteq \bd \bO$ is a closed and Ahlfors--David regular portion of its boundary, and that $\bO$ is locally uniform near $\bd\bO \setminus \bD$. The proof heavily relies on \cite{Laplace-Extrapolation}, which can essentially be used as a black box, but nonetheless the reader is advised to keep a copy of that paper handy.

\subsection{The idea of Axelsson--Keith--McIntosh}
\label{Subsec: The Idea of AKM}

On the Hilbert space $H \coloneqq \L^2(\bO)^m \times \L^2(\bO)^{dm} \times \L^2(\bO)^m$ introduce the closed operators with maximal domain
\begin{align}
\label{Special AKM operators}
 \bGamma \coloneqq \begin{bmatrix} 0 & 0 & 0 \\ 1 & 0 & 0 \\ \nabla_{\bD} & 0 & 0 \end{bmatrix}, \qquad
\boldsymbol{B}_1 \coloneqq \begin{bmatrix} \boldsymbol{1} & 0 & 0 \\ 0 & 0 &0 \\ 0 & 0 & 0\end{bmatrix}, \qquad
\boldsymbol{B}_2 \coloneqq \begin{bmatrix} 0 & 0 & 0 \\ 0 & \bbd & \bc \\ 0 & \bb & \bA  \end{bmatrix},
\end{align}
where $\nabla_{\bD}$ was defined in \eqref{eq: nablaD} and $-\div_{\bD}$ is its adjoint. Then define the \emph{perturbed Dirac operator} $\bPiB \coloneqq \bGamma + \boldsymbol{B}_1 \bGamma^* \boldsymbol{B}_2$ on $\dom(\bPiB) \coloneqq \dom(\bGamma) \cap \dom(\boldsymbol{B}_1 \bGamma^* \boldsymbol{B}_2)$. It follows that
\begin{align}
\label{Special AKM operators squared}
\bPiB = \begin{bmatrix} 0 & 0 & 0  \\ 1 & 0 & 0 \\ \nabla_{\bD} & 0 & 0 \end{bmatrix} 
+ \begin{bmatrix} 0 & 1 & -\div_{\bD} \\ 0 & 0 & 0 \\ 0 & 0 & 0 \end{bmatrix} \begin{bmatrix} 0 & 0 & 0 \\ 0 & \bbd & \bc \\ 0 & \bb & \bA  \end{bmatrix}, \qquad 
\bPiB^2 = \begin{bmatrix} \bL & 0 & 0  \\ 0 & * & * \\ 0 & * & * \end{bmatrix}.
\end{align}
Here, coefficients are identified with the corresponding multiplication operators and, owing to \eqref{eq: L as composition}, the second order operator $\bL$ with correct domain appears. The precise structure of the asterisked entries is not needed. The operators in \eqref{Special AKM operators} have the following properties.
\begin{enumerate}
	\item[\quad (H1)] $\bGamma$ is \emph{nilpotent}, that is, closed, densely defined and satisfies $\Rg(\bGamma) \subseteq \Ke(\bGamma)$.
	\item[\quad (H2)] $\boldsymbol{B}_1$ and $\boldsymbol{B}_2$ are defined on $H$. There exist $\kappa_i, K_i \in (0,\infty)$, $i=1,2$, such that 
	\begin{align}
	\Re ( \boldsymbol{B}_1 U \SP U)_2 &\geq \kappa_1 \|U\|_2^2 \qquad \mathrlap{(U \in \Rg(\bGamma^*)),} \\
	\Re ( \boldsymbol{B}_2 U \SP U)_2 &\geq \kappa_2 \|U\|_2^2 \qquad \mathrlap{(U \in \Rg(\bGamma)),} \\
	\|\boldsymbol{B}_i U\|_2 &\leq K_i \|U\|_2 \qquad \mathrlap{\!(U \in H).}
	\end{align}
	\item[\quad (H3)] $\boldsymbol{B}_2 \boldsymbol{B}_1$ maps $\Rg(\bGamma^*)$ into $\Ke(\bGamma^*)$ and $\boldsymbol{B}_1 \boldsymbol{B}_2$ maps $\Rg(\bGamma)$ into $\Ke(\bGamma)$.
\end{enumerate}
Indeed, in (H2) we can take $\kappa_1 = 1$ and $\kappa_2 = \lambda$, see \eqref{eq: a rewritten} and also \eqref{Garding}. Abstract Hilbert space theory therefore yields that $\bPiB$ is bisectorial~\cite[Prop.~2.5]{AKM-QuadraticEstimates} and that the \emph{unperturbed Dirac operator} $\boldsymbol{\Pi} \coloneqq \bGamma + \bGamma^\ast$ is self-adjoint~~\cite[Cor.~4.3]{AKM-QuadraticEstimates}. 

Suppose that $\bPiB$ even has a bounded $\H^\infty$-calculus on $\cl{\Rg(\bPiB)} \subseteq H$. Then $\dom(\sqrt{\bPiB^2}) = \dom(\bPiB)$ follows with equivalent homogeneous graph norms, see Example~\ref{Ex: Abstract Kato}. In both operators the first component acts independently of the others and is defined on $\dom(\sqrt{\bL})$ and $\W^{1,2}_{\bD}(\bO)^m$, 
respectively. This gives $\dom(\sqrt{\bL}) = \W^{1,2}_{\bD}(\bO)^m$ and the Kato estimate 
\begin{align*}
\|u\|_2 + \|\nabla u\|_2 
\approx \left\|\bPiB \begin{bmatrix} u \\ 0 \\ 0 \end{bmatrix} \right\|_2 
\approx \left \|\sqrt{\bPiB^2} \begin{bmatrix} u \\ 0 \\ 0 \end{bmatrix} \right\|_2
= \|\sqrt{\bL} u\|_2 \qquad (u \in \W^{1,2}_{\bD}(\bO)^m).
\end{align*} 
Implicit constants depend on $\bL$ only through the bound for the $\H^\infty$-calculus for $\bPiB$. In order to prove Theorem~\ref{Thm: main kato result} under the additional interior thickness assumption we have to argue that $\bPiB$ indeed has a bounded $\H^\infty$-calculus with a bound that depends on $\bL$ only through its coefficient bounds, or what is equivalent thereto by McIntosh's theorem~\ref{Thm: McIntosh}, that $\bPiB$ satisfies the quadratic estimate
\begin{align}
\label{eq: QE PiB}
\int_0^\infty \|t \bPiB (1+t^2 \bPiB^2)^{-1} U\|_2^2 \; \frac{\d t}{t} \approx \|U\|_2^2 \qquad \mathrlap{(U \in \cl{\Rg(\bPiB)})}
\end{align}
with the same dependency of the implicit constants.

\subsection{Quadratic estimates for Dirac operators}
\label{Subsec: QE for Dirac}

There are general frameworks of perturbed Dirac operators~\cite{AKM, AKM-QuadraticEstimates, Laplace-Extrapolation}, each of which starts from a triple of operators $(\bGamma, \boldsymbol{B}_1, \boldsymbol{B}_2)$ on $H$ that verifies (H1) - (H3). Additional hypotheses (H4) - (H7) on the operators and certain geometric assumptions are required in order to obtain \eqref{eq: QE PiB}. 

We will soon see that the operators in \eqref{Special AKM operators} verify (H4) - (H7) from \cite{Laplace-Extrapolation}. For the time being, we take that for granted and discuss the geometric assumptions. There are four of them~\cite[Ass.~2.1]{Laplace-Extrapolation}:
\begin{enumerate}
	\item[\quad ($\bO$)] Comparability $|B \cap \bO| \approx |B|$ holds uniformly for all balls $B$ of radius $\rad(B) \leq 1$ centered in $\bO$.
	\item[\quad ($\bd \bO$)] Comparability $\cH^{d-1}(B \cap \bd \bO) \approx r^{d-1}$ holds uniformly for all balls $B$ of radius $\rad(B) \leq 1$ centered in $\bd \bO$.
	\item[\quad ($V$)] Multiplication by $\C_0^\infty(\R^d)$-functions maps the form domain $V =\W_{\bD}^{1,2}(\bO)^m$ into itself and there exists a bounded extension operator $E: V \to \W^{1,2}(\R^d)^m$.
	\item[\quad ($\alpha$)] For some $\alpha \in (0,1)$ the complex interpolation space $[\L^2(\bO)^m, V]_\alpha$ coincides with $\W^{\alpha,2}(\bO)^m$ up to equivalent norms.
\end{enumerate}
In \cite{Laplace-Extrapolation} the terminology \emph{domain} was used for \emph{non-empty proper open subset} and $\bO$ was named $\Omega$. Their central result~\cite[Thm.~3.2]{Laplace-Extrapolation} is as follows.

\begin{theorem}
\label{Thm: EHT main result}
Under the structural assumptions (H1) - (H7) and the geometric assumptions $(\bO)$, $(\bd \bO)$, $(V)$, $(\alpha)$ the operator $\bPiB$ is bisectorial and satisfies \eqref{eq: QE PiB}. Implicit constants depend on $\boldsymbol{B}_1$, $\boldsymbol{B}_2$ only through the parameters quantified in (H2).
\end{theorem}

Let us see how this relates to our geometric assumptions. Clearly ($\bO$) is just the same as \eqref{ITC}. Proposition~\ref{Prop: Interpolation} yields $(\alpha)$ for every $\alpha \in (0,\frac{1}{2})$. (Complex interpolation of (at most) countable products of spaces works componentwise~\cite[Sec.~1.18.1]{Triebel}). Componentwise application of Theorem~\ref{Thm: W12D extension} furnishes the extension operator in ($V$) and the stability property follows since multiplication by $\C_0^\infty(\R^d)$-functions is bounded on $\W^{1,2}(\bO)^m$ and maps the dense subset $\C_{\bD}^\infty(\bO)^m \subseteq V$ into itself. 

Our major point here is that $(\bd \bO)$ can be replaced by the significantly weaker assumption that $\bd \bO$ is porous. By Corollary~\ref{Cor: Boundary porous} our boundary $\bd \bO$ has the latter property. 

Fortunately, the reader does not have to go through all of \cite{Laplace-Extrapolation} in order to see why this relaxation of geometric assumptions works. Indeed, as is clearly stated in that paper before Lemma~7.6, $(\bd \bO)$ is used only once, namely to ensure validity of the following lemma (with a constant $\hat{\eta} > 0$ that happens to be $1$ under ($\bd \bO$)). Compare also with their Corollary~7.8.

\begin{lemma}
\label{Lem: bdO lemma from EHT}
If $\bO \subseteq \R^d$ is open and satisfies $(\bd \bO)$, then for each $r_0, t_0 > 0$ there exists $C>0$ and $\hat{\eta}>0$ such that
\begin{align}
\label{eq: bdO lemma from EHT}
\Big|\Big\{x \in \bO: |x-x_0| < r, \, \dist(x, \R^d \setminus \bO) \leq tr \Big\}\Big| \leq C t^{\hat{\eta}} r^d 
\end{align}
for all $x_0 \in \cl{\bO}$, $r \in (0, r_0]$ and $t \in (0,t_0]$.
\end{lemma}

Hence, our only task in relaxing $(\bd \bO)$ is to reprove that lemma under the mere assumption that $\bd \bO$ is porous, which we will do now.

\begin{proof}[Proof of Lemma~\ref{Lem: bdO lemma from EHT} assuming only that $\bd \bO$ is porous]
Let $E$ be the set in \eqref{eq: bdO lemma from EHT}. For $t \geq 1$ the trivial bound $|E| \leq |B(x_0,r)| \lesssim r^d$ is enough. Thus, we can assume $t < 1$. 
	
To each $x \in E$ there corresponds some $x_\partial \in \bd \bO$ with $|x-x_\partial| \leq tr$. Since different $x \in E$ are at distance less than $2r$ from each other, we can pick a ball $B$ of radius $4r$ centered in $\bd \bO$ that contains all $x_\partial$. Temporarily assume $4r \leq 1$. Then we can use Lemma~\ref{Lem: porous covering property} and obtain $C \geq 1$ and $0<s<d$ such that $B \cap \bd \bO$ can be covered by at most $C(4/t)^s$ balls $B_i$ of radius $tr$ centered in $\bd \bO$. Hence, each $x \in E$ is contained in one of the balls $2B_i$ and we conclude
\begin{align*}
|E| \lesssim (2tr)^d \#_i \leq C 2^d 4^s t^{d-s} r^d.
\end{align*}
In the case $4r > 1$ we have the same type of covering property for $B \cap \bd \bO$: Indeed, first we use the Vitali lemma to cover $B \cap \bd \bO$ by $10^d (4r)^d \leq 10^d (4r_0)^d$ balls of radius $1$ centered in $\bd \bO$ and then we use Lemma~\ref{Lem: porous covering property} with balls of radius $\frac{t}{4} \leq tr$. This affects the value of $C$ but we can still take $\hat{\eta} \coloneqq d-s$.
\end{proof}

The upshot is that Theorem~\ref{Thm: EHT main result} yields the quadratic estimates~\eqref{eq: QE PiB} for $\bPiB$ and hence the proof is complete once we have verified (H4) - (H7). 

\subsection{The additional Dirac operator hypotheses}
\label{Subsec: H1 - H7}

Here are the additional hypotheses of \cite[Sect.~5]{Laplace-Extrapolation} that the operators in \eqref{Special AKM operators} have to verify. It is convenient to set $n \coloneqq m(d+2)$ for the number of components of a function in $H =  \L^2(\bO)^m \times \L^2(\bO)^{dm} \times \L^2(\bO)^m$.

\begin{enumerate}
	\item[\quad (H4)] $\boldsymbol{B}_1, \boldsymbol{B}_2$ are multiplication operators with functions in $\L^\infty(\bO; \cL(\IC^n))$.
	\item[\quad (H5)] For every $\varphi \in \C_0^\infty(\R^d)$ the associated multiplication operator $M_\varphi$ maps $\dom(\bGamma)$ into itself. The commutator $\bGamma M_\varphi - M_\varphi \bGamma$ with domain $\dom(\bGamma)$ acts via multiplication by some $c_\varphi \in \L^\infty(\bO; \cL(\IC^n))$ and its components satisfy $|c_\varphi^{i,j}(x)| \lesssim |\nabla \varphi(x)|$ for an implicit constant that does not depend on $\varphi$.
	\item[\quad (H6)] For every open ball $B$ centered in $\bO$, and for all $U \in \dom(\bGamma)$, $V \in \dom(\bGamma^*)$ both with compact support in $B \cap \bO$ it follows that
	\begin{align*}
	\bigg\lvert \int_{\bO} \bGamma U  \d x  \bigg\rvert &\lesssim \lvert B \rvert^{\frac{1}{2}} \|U\|_2, \\
	\bigg\lvert \int_{\bO} \bGamma^* V \d x \bigg\rvert &\lesssim \lvert B \rvert^{\frac{1}{2}} \|V\|_2.
	\end{align*}
	\item[\quad (H7)] There exist $\beta , \gamma \in (0, 1]$ such that the fractional powers of $\boldsymbol{\Pi} = \bGamma + \bGamma^*$ satisfy
	\begin{align*}
	\|U\|_{[H , \W^{1,2}_{\bD}(\bO)^{n}]_{\beta}} &\lesssim \|( \boldsymbol{\Pi}^2 )^{\beta/2}U\|_2 \qquad \mathrlap{(U \in \Rg(\bGamma^*) \cap \dom(\boldsymbol{\Pi}^2)),} \\
	\|V\|_{[H , \W^{1,2}_{\bD}(\bO)^{n}]_{\gamma}} &\lesssim \|(\boldsymbol{\Pi}^2 )^{\gamma/2}V\|_2 \qquad \mathrlap{(V \in \Rg(\bGamma) \cap \dom(\boldsymbol{\Pi}^2)),}
	\end{align*}
	where $[\cdot \,,\cdot]$ denotes again the complex interpolation bracket.
\end{enumerate}

In (H5) we have $\dom(\bGamma) = \W_{\bD}^{1,2}(\bO)^m \times \L^2(\bO)^{dm} \times \L^2(\bO)^m$. The mapping property follows from $(V)$ and the commutator assertion from the product rule. By duality, (H5) also holds for $\bGamma^*$ with commutators $\bGamma^* M_\varphi - M_\varphi \bGamma^* = -c_{\cl{\varphi}}^*$.

For (H6), take a unit vector $e \in \IC^{n}$ and let $\varphi \in \C_0^\infty(\bO)$ be valued in $[0,1]$ with $\varphi = 1$ on $\supp(U)$. We have $|U \cdot \bGamma^*(\varphi e)| \leq |U|$ since $\div_{\bD}(\varphi e) = \div(\varphi e) = 0$ on $\supp(U)$, see also Section~\ref{Sec: Elliptic operator}. Moreover, $\varphi \bGamma U = \bGamma U$ follows from (H5) using $\nabla \varphi = 0$ on $\supp(U)$. We obtain
\begin{align*}
\bigg|\int_{\bO} \bGamma U \cdot e  \d x \bigg| = \bigg|\int_{\bO} U\cdot \bGamma^*(\varphi e)  \d x \bigg|\leq \int_{\bO} |U|  \d x \leq |B|^{1/2} \|U\|_2,
\end{align*}
which suffices since $e$ was arbitrary. For $V$ we simply switch the roles of $\bGamma$ and $\bGamma^*$.

In (H7) we take $\beta = 1$. By choice of $U$ we conclude $U \in \W_{\bD}^{1,2}(\bO)^m \times \{0\} \times \{0\}$, so we obtain from the bounded $\H^\infty$-calculus for the self-adjoint operator $\boldsymbol{\Pi}$ and Example~\ref{Ex: Abstract Kato} that
\begin{align}
\label{eq1: H7 verification}
\|U\|_{\W^{1,2}_{\bD}(\bO)^n} = \|\boldsymbol{\Pi} U \|_2 \approx \|(\boldsymbol{\Pi}^2)^{1/2}U \|_2.
\end{align}
We take $\gamma \in (0,\eps)$ with $\eps$ as in Theorem~\ref{Thm: Laplace}. As $\boldsymbol{\Pi}^2$ corresponds to $\bA=\boldsymbol{1}$, $\bb=\bc^t=\boldsymbol{0}$, $\bbd=1$ in \eqref{Special AKM operators squared}, we discover $-\Delta_{\bD}+1$ in the upper left corner and obtain with
\begin{align}
\label{eq2: H7 verification}
 W \coloneqq \begin{bmatrix}
v \\ 0 \\ 0
\end{bmatrix}
\quad \text{that} \quad
 V 
= 
\begin{bmatrix}
0 \\ v \\ \nabla_{\bD} v
\end{bmatrix}
= \boldsymbol{\Pi} \, W,
\end{align}
where $v \in \dom(-\Delta_{\bD}+1)$. We conclude
\begin{align}
\|V\|_{[H , \W^{1,2}_{\bD}(\bO)^n]_{\gamma}}
\approx \|V\|_{\W^{\gamma,2}(\bO)^n}
&\lesssim \|v\|_{\W^{1+\gamma,2}_{\bD}(\bO)}\\
&\approx \|(-\Delta_{\bD}+1)^{1/2+\gamma/2}v\|_2
= \|(\boldsymbol{\Pi}^2)^{1/2+\gamma/2} W\|_2,
\end{align}
where the first step is due to Proposition~\ref{Prop: Interpolation}, the second step uses Lemma~\ref{Lem: mapping property gradient} and the third one follows from Theorem~\ref{Thm: Laplace}. Using~\eqref{eq1: H7 verification} and the version of Example~\ref{Ex: Fractional powers} for bisectorial operators, the last term compares to $\|\boldsymbol{\Pi} (\boldsymbol{\Pi}^2)^{\gamma/2} W\|_2 =\|(\boldsymbol{\Pi}^2)^{\gamma/2} V\|_2$. 

The proof of Theorem~\ref{Thm: main kato result} is now complete under the additional assumption that the underlying open set satisfies the interior thickness condition \eqref{ITC}. Note that here $\bO$ takes the role of $O$ in Theorem~\ref{Thm: main kato result}. At this point in the proof, our result already fully covers all earlier results from the literature.
\section{Elimination of the interior thickness condition}
\label{Sec: Fattening}
In this section we complete the proof of Theorem~\ref{Thm: main kato result} with the strategy sketched in Step~3 from the introduction. In the whole section we work with two triples of domain, Dirichlet part and elliptic system: $(O,D,L)$ will be satisfying the assumptions from Theorem~\ref{Thm: main kato result} and for $(\bO, \bD, \bL)$ we start with $\bO \subseteq \R^d$ open and $\bD \subseteq \bd \bO$ closed, but further properties including interior thickness of $\bO$ will be added in the course of the proof. 

\subsection{Localization of the functional calculus to invariant open subsets}
\label{Subsec: Localization to components}

\begin{definition}
\label{Def: good projection}
	 A \emph{good projection} is an orthogonal projection $\Proo$ on $\L^2(\bO)$ that commutes with bounded multiplication operators and with $\nabla_{\bD}$ on $\W^{1,2}_{\bD}(\bO)$. In this case the inclusion map from the image $\Proo \L^2(\bO)$  of $\Proo$ back into $\L^2(\bO)$ is denoted by $\Proo^*$.
\end{definition}
Note that the inclusion map from Definition~\ref{Def: good projection} coincides with the adjoint of $\Proo: \L^2(\bO) \to \Proo \L^2(\bO)$, which justifies the usage of the symbol $\Proo^\ast$. We let good projections act componentwise on $\L^2(\bO)^m$. The concrete example the reader should have in mind is that $\Proo$ is the multiplication by the characteristic function of the union of connected components of $\bO$ if the  latter fulfills the geometric requirements from Theorem~\ref{Thm: main kato result}. We shall come back to that. In fact, we will only work with two good projections later on, but we believe that the more general localization procedure that we are going to construct in this section could prove useful elsewhere.

Let $(\Proo_i)_{i\in I}$, $I \subseteq \IN$, be a family of pairwise orthogonal good projections which decomposes $\L^2(\bO)^m$ in the sense that 
\begin{align}
\L^2(\bO)^m \cong \bigotimes_i \Proo_i \L^2(\bO)^m \quad \text{via the topological isomorphism} \quad S: U \mapsto (\Proo_i U)_i.
\end{align}
The $\ell^2$-tensor notation has been introduced in the introduction. We remind the reader that in such a context we also use $\ell^2$-tensors of operators who act componentwise on their natural domain, see again the introduction. Since a good projection commutes with $\nabla_{\bD}$ it is also bounded on $\W^{1,2}_{\bD}(\bO)$. Consequently, we have 
\begin{align}
\label{W12D good projection splitting}
\W^{1,2}_{\bD}(\bO)^m \cong \bigotimes_i \Proo_i \W^{1,2}_{\bD}(\bO)^m
\end{align}
via the same isomorphism $S$ as before.

\begin{lemma}
\label{Lem: Pi_B commutes with good projections}
	Let $\Proo$ be a good projection. Then $\Proo$ commutes with $\bL$ in the sense that $\Proo \bL \subseteq \bL \Proo$.
\end{lemma}
\begin{proof}
Recall the block decomposition~\eqref{eq: L as composition} of $\bL$. By definition, $\Proo$ commutes with the second and third block of that decomposition. By duality and self-adjointness of the projection, $\Proo$ also commutes with the first block and therefore with $\bL$.
\end{proof}

The lemma shows that $\bL \Proo_i^\ast$ is an operator in $\Proo_i \L^2(\bO)^m$. More precisely, it is the part of $\bL$ in $\Proo_i \L^2(\bO)^m$ with maximal domain $\Proo_i \dom(\bL)$. An abstract property of functional calculi stated in Proposition~\ref{Prop: FC transformation} allows us to pull the projections in and out of the functional calculus, that is 
\begin{align}
\label{eq: calculation rules projections}
\Proo_i f(\bL) \subseteq f(\bL) \Proo_i \qquad \text{and} \qquad f(\bL \Proo_i^\ast)=f(\bL) \Proo_i^\ast.
\end{align}
As above we get in particular $\dom(f(\bL \Proo_i^\ast)) = \Proo_i \dom(f(\bL))$. We use these calculation rules freely in order to give the following decomposition of the functional calculus for $\bL$ via good projections.

\begin{proposition}
\label{Prop: f(L) decomposition}
	One has
	\begin{align} \label{eq: f(L) decomposition}
		f(\bL) = S^{-1} \Bigl[\bigotimes_i f(\bL \Proo_i^\ast) \Bigr] S \quad\text{with}\quad \dom(f(\bL)) = S^{-1} \Bigl(\bigotimes_i \dom(f(\bL \Proo_i^\ast)) \Bigr),
	\end{align}
	where the equality of spaces is with equivalent norms.
\end{proposition}
\begin{proof}
If $u\in \dom(f(\bL))$, then $\Proo_i u \in \dom(f(\bL \Proo_i^\ast))$ and $\Proo_i f(\bL ) u = f(\bL \Proo_i^\ast) Q_i u$ hold for all $i$ according to \eqref{eq: calculation rules projections}. We conclude
\begin{align}
 S f(\bL) u
 = (\Proo_i f(\bL) u)_i 
 = (f(\bL \Proo_i^\ast) Q_i u)_i
 = \Big[\bigotimes_i f(\bL \Proo_i^\ast) \Big] Su
\end{align}
and the inclusion ``$\subseteq$" of operators in \eqref{eq: f(L) decomposition} follows as $S$ is an isomorphism. Conversely, let $(u_i)_i \in \dom(\bigotimes_i f(\bL \Proo_i^\ast))$. Then $u \coloneqq \sum_i u_i$ converges in the $\ell^2$ sense and we have $S u = (u_i)_i$. It remains to prove $u \in \dom(f(\bL))$. From the second identity in~\eqref{eq: calculation rules projections} we get $u_i \in \dom(f(\bL))$ for every $i$ as well as
\begin{align}
 f(\bL) \Big[\sum_{i \in I \cap \{0,\ldots,n\}} u_i \Big] = \sum_{i \in I \cap \{0,\ldots,n\}} f(\bL \Proo_i^\ast) u_i \mathrlap{\qquad (n \in \IN).}
\end{align}
In the limit as $n \to \infty$ the term on the right-hand side converges by definition of the domain of the tensorized operator and $\sum_{i \in I \cap \{0,\ldots,n\}} u_i$ tends to $u$. Since $f(\bL)$ is a closed operator, we conclude $u \in \dom(f(\bL))$.
\end{proof}

\subsection{Projections coming from indicator functions}
\label{Sec: Projections coming from indicator functions}

From now on we assume that $\bO$ is locally uniform near $\bN\coloneqq \bd \bO \setminus \bD$ and that there is a decomposition $\bO = \bigcup_i O_i$, where the $O_i$ are pairwise disjoint open sets. Since $O_i$ is open and closed in $\bO$, it follows $\bd O_i \subseteq \bd \bO$. We put $D_i \coloneqq \bD\cap \bd O_i$. We write $\Pro_i$ for the orthogonal projection on $\L^2(\bO)$ induced by multiplication with $\mathds{1}_{O_i}$. We also use the zero-extension operators 
\begin{align}
\Ext_i: \L^2(O_i) \to \Pro_i \L^2(\bO).
\end{align}
These are unitary with $\Ext_i^*$ the pointwise restriction of functions to $O_i$. This allows us to identify $\L^2(O_i)^m$ with $\Pro_i \L^2(\bO)^m$. 

We start by investigating how this identification extends to $\W^{1,2}_{\bD}(\bO)^m$.

\begin{lemma}
\label{Lem: test function zero extension}
	Let $\varphi \in \C_{D_i}^\infty(O_i)^m$. If $E\varphi\in \C_{D_i}^\infty(\R^d)^m$ is any extension, then
	\begin{align}
		\dist(\supp(E \varphi) \cap O_i, \bO\setminus O_i)>0,
	\end{align}
	and $\Ext_i \varphi \in \C_{\bD}^\infty(\bO)^m$ with $\nabla \Ext_i \varphi = \Ext_i \nabla \varphi$.
\end{lemma}
\begin{proof}
	For the first claim let $x \in \supp(E \varphi) \cap O_i$ and let $z' \in \cl{\bO \setminus O_i} = \cl{\bO} \setminus O_i$ realize the distance  of $x$ to $\bO\setminus O_i$. Hence, we can pick some $z \in \bd O_i$ on the line segment connecting $x$ and $z'$. 
	First, consider the case that $z \in D_i$. Then $$\dist(x, \bO\setminus O_i) = |x-z'| \geq |x-z| \geq \dist(\supp E\varphi, D_i) > 0.$$ Otherwise, we are in the case $z\in \bN$. If $\dist(x, \bO\setminus O_i) \geq \frac{\delta}{2}$, then we are done, so let us assume $\dist(x, \bO\setminus O_i) < \frac{\delta}{2}$, so that in particular $|x-z| < \frac{\delta}{2}$. Then there is $y \in \bO\setminus O_i$ such that $|x-y| < \frac{\delta}{2}$. This gives $x,y \in \bN_\delta$ and by Definition~\ref{Def: locally eps-delta} we can join $x$ and $y$ by a continuous path in $\bO$. But as $x \in O_i$ and $y\in \bO\setminus O_i$, this path has to cross $\bd O_i \subseteq \bd \bO$, which leads to a contradiction. Consequently, with $\rho\coloneqq \min(\dist(\supp E\varphi, D_i), \frac{\delta}{2})$ we get
	\begin{align}
	\dist(\supp(E \varphi) \cap O_i, \bO\setminus O_i) \geq \rho.
	\end{align}
	For the second claim we fix a smooth function $\chi$ equal to $1$ outside the $\rho$-neighborhood of $\bO\setminus O_i$ and equal to $0$ on the respective $\frac{\rho}{2}$-neighborhood. Then $\chi E \varphi \in \C_0^\infty(\R^d)^m$ vanishes on $\bO \setminus O_i$, whereas on $O_i$ we have $\chi E \varphi = E \varphi = \varphi$ by the choice of $\rho$. We conclude $\Ext_i \varphi = (\chi E\varphi)|_{\bO}$ and $\nabla \Ext_i \varphi = \Ext_i \nabla \varphi$ on $\bO$. Finally, let $x \in \bD$. Then either there is $j \neq i$ with $\dist(x, O_j) < \frac{\rho}{2}$, in which case we have $\dist(x, \bd \bO \setminus O_i) < \frac{\rho}{2}$ and hence $\chi = 0$ in a neighborhood of $x$. Else, we have $\dist(x, O_j) \geq \frac{\rho}{2}$ for all $j \neq i$, so that $x \in \bD \subseteq \bd \bO$ implies $x \in \bD \cap \bd O_i = D_i$ and hence $E \varphi = 0$ holds near $x$. This proves $\chi E \varphi \in \C_{\bD}^\infty(\R^d)^m$. Now, $\Ext_i \varphi \in \C_{\bD}^\infty(\bO)^m$ follows by restriction to $\bO$.
\end{proof}

\begin{proposition}
\label{Prop: Pi good projections}
	The $\Pro_i$ are good projections that satisfy $\Pro_i \W^{1,2}_{\bD}(\bO) = \Ext_i \W^{1,2}_{D_i}(O_i)$.
\end{proposition}
\begin{proof}
	By Lemma~\ref{Lem: test function zero extension} we know that $\Ext_i$ is an isometry from the dense subset $\C^\infty_{D_i}(O_i)$ of $\W^{1,2}_{D_i}(O_i)$ into $\W^{1,2}_{\bD}(\bO)$. Therefore, $\Ext_i$ is $\W^{1,2}_{D_i}(O_i)\to \W^{1,2}_{\bD}(\bO)$ bounded. Since $\Ext_i$ maps into the range of $\Pro_i$, we arrive at $\Ext_i \W^{1,2}_{D_i}(O_i) \subseteq \Pro_i \W^{1,2}_{\bD}(\bO)$.
	
	On the other hand, take $\varphi \in \C^\infty_{\bD}(\bO)$. Then we have  $\varphi|_{O_i} \in \C^\infty_{D_i}(O_i)$ and
	\begin{align}
		\Pro_i \varphi = \Ext_i (\varphi|_{O_i}) \label{eq: Pi decomposition}
	\end{align} 
	is in $\C_{\bD}^\infty(\bO)$ due to Lemma~\ref{Lem: test function zero extension} with gradient 
	\begin{align}
		\nabla \Pro_i \varphi  = \Ext_i \nabla (\varphi|_{O_i}) = \Ext_i (\nabla  \varphi)|_{O_i} = \Pro_i \nabla \varphi.
	\end{align}
	As a consequence, we get that $\Pro_i$ is bounded on $\C^\infty_{\bD}(\bO)$ for the $\W^{1,2}_{\bD}(\bO)$-norm. By density, $\Pro_i$ is bounded on $\W^{1,2}_{\bD}(\bO)$ and the identity above extends to the same space, thereby showing that $\Pro_i$ is a good projection. We use this to have a second look on identity~\eqref{eq: Pi decomposition}. The left-hand side is bounded on $\W^{1,2}_{\bD}(\bO)$ by the foregoing argument and the right-hand side maps into $\Ext_i \W^{1,2}_{D_i}(O_i)$ by the very first step of this proof. We conclude $\Pro_i \W^{1,2}_{\bD}(\bO) \subseteq \Ext_i \W^{1,2}_{D_i}(O_i)$ by continuity.
\end{proof}

From the preceding proposition and \eqref{W12D good projection splitting} we get the decomposition 
\begin{align}
\label{W12D desired splitting}
\W^{1,2}_{\bD}(\bO)^m \cong \bigotimes_i \Pro_i \W^{1,2}_{\bD}(\bO)^m = \bigotimes_i \Ext_i \W^{1,2}_{D_i}(O_i)^m
\end{align}
via the usual isomorphism $S$. Analogously to Section~\ref{Sec: Elliptic operator}, we introduce in $\L^2(O_i)^m$ the divergence form operator $L_i$ with coefficients $\boldsymbol{A}|_{O_i}$, $\boldsymbol{b}|_{O_i}$, $\boldsymbol{c}|_{O_i}$, $\boldsymbol{d}|_{O_i}$ corresponding to the sesquilinear form
\begin{align}
a_i: \W^{1,2}_{D_i}(O_i)^m \times \W^{1,2}_{D_i}(O_i)^m \to \IC, \qquad a_i(u,v) = \int_{O_i} \; \begin{bmatrix} \boldsymbol{d} & \boldsymbol{c} \\ \boldsymbol{b} & \boldsymbol{A} \end{bmatrix} \begin{bmatrix} u \\ \nabla u \end{bmatrix}  \cdot \cl{\begin{bmatrix} v \\ \nabla v \end{bmatrix}} \d x.
\end{align}
We will see momentarily that this operator is unitarily equivalent to $\bL \Pro_i^\ast$. As in Section~\ref{Sec: Mapping properties for Laplacian}, the key lies in showing unitary equivalence for the gradients. 

\begin{lemma}
\label{Lem: similarity gradient}
	The operators $\nabla_{D_i}$ and $\nabla_{\bD} \Pro_i^\ast$ are unitarily equivalent via $\Ext_i \nabla_{D_i} = \nabla_{\bD} \Pro_i^\ast \Ext_i$.
\end{lemma}

\begin{proof}
	The two operators have the same domain $\W^{1,2}_{D_i}(O_i)$ by definition and Proposition~\ref{Prop: Pi good projections}. Moreover, they have the same action as both appearing gradients are restrictions of the respective distributional gradient, which commutes with the zero extension operator. 
\end{proof}

\begin{proposition}
\label{Prop: similarity L_i and part of L}
	The operators $L_i$ and $\bL \Pro_i^\ast$ are unitarily equivalent via $\Ext_i L_i = \bL \Pro_i^\ast \Ext_i$.
\end{proposition}

\begin{proof}
First of all, we note that $\dom(L_i)$ and $\dom(\bL \Pro_i^\ast \Ext_i)$ are both subsets of $\W^{1,2}_{D_i}(O_i)^m$. For the first operator this holds by definition, whereas for the second one it follows from $\dom(\bL) \subseteq \W^{1,2}_{\bD}(\bO)^m$ and Proposition~\ref{Prop: Pi good projections}. Next, let $u,v \in \W^{1,2}_{D_i}(O_i)^m$ and let $Ev \in \W^{1,2}_{\bD}(\bO)^m$ be any extension of $v$. Lemma~\ref{Lem: similarity gradient} yields
\begin{align*}
a_i(u,v)
&= \int_{\bO} \; \begin{bmatrix} \boldsymbol{d} & \boldsymbol{c} \\ \boldsymbol{b} & \boldsymbol{A} \end{bmatrix} \begin{bmatrix} \Ext_i u \\ \Ext_i \nabla u \end{bmatrix}  \cdot \cl{\begin{bmatrix} Ev \\ \nabla Ev \end{bmatrix}} \d x \\
&= \int_{\bO} \; \begin{bmatrix} \boldsymbol{d} & \boldsymbol{c} \\ \boldsymbol{b} & \boldsymbol{A} \end{bmatrix} \begin{bmatrix} \Pro_i^\ast \Ext_i u \\ \nabla (\Pro_i^\ast \Ext_i u) \end{bmatrix}  \cdot \cl{\begin{bmatrix} Ev \\ \nabla Ev \end{bmatrix}} \d x
= \boldsymbol{a}(\Pro_i^\ast \Ext_i u, Ev).
\end{align*}
If $u \in \dom(L_i)$, then the left-hand side becomes $(L_i u\,|\,v)_{\L^2(O_i)^m} = (\Ext_i  L_iu\,|\,Ev)_{\L^2(\bO)^m}$. Since $Ev$ can be any function in $\W^{1,2}_{\bD}(\bO)^m$, we conclude that $\Pro_i^\ast \Ext_i u \in \dom(\bL)$ with $\bL \Pro_i^\ast \Ext_i u = \Ext_i L_i u$. Conversely, let $u \in \dom(\bL \Pro_i^\ast \Ext_i)$. We take $Ev = \Ext_i v$, which is admissible by Proposition~\ref{Prop: Pi good projections}, and the right-hand side becomes $(\bL \Pro_i^\ast \Ext_i u \,|\, \Ext_i v)_{\L^2(\bO)^m} = (\Ext_i^\ast \bL \Pro_i^\ast \Ext_i u \,|\, v )_{\L^2(O_i)^m}$. This proves $u \in \dom(L_i)$ with $L_i u = \Ext_i^\ast \bL \Pro_i^\ast \Ext_i u$.
\end{proof}

Let us summarize the situation. The set $\bO$ is locally uniform near $\bd \bO \setminus \bD$ and can be decomposed into pairwise disjoint open sets $O_i$. The divergence form operator $\bL$ is given by coefficients $\boldsymbol{A}$, $\boldsymbol{b}$, $\boldsymbol{c}$, $\boldsymbol{d}$ on $\bO$ and $L_i$ is the divergence form operator on $O_i$ whose coefficients are obtained by restricting the coefficients of $\bL$ to $O_i$, and which is subject to a vanishing trace condition on $D_i = \bD \cap \bd O_i$.
Combining all intermediate steps, we derive the following correspondence.

\begin{proposition}
\label{Prop: Regularity decomposition for the square root}
	The following are equivalent:
	\begin{enumerate}
		\item $\dom(\sqrt{\bL})=\W^{1,2}_{\bD}(\bO)^m$ with $\|\sqrt{\bL} u\|_2 \approx \|u\|_{\W^{1,2}(\bO)^m}$,
		\item $\dom(\sqrt{L_i})=\W^{1,2}_{D_i}(O_i)^m$ with $\|\sqrt{L_i} u\|_2 \approx \|u\|_{\W^{1,2}(O_i)^m}$ for all $i$, where the implicit constants are independent of $i$.
	\end{enumerate}
\end{proposition}

\begin{proof}
Proposition~\ref{Prop: similarity L_i and part of L} gives $\bL \Pro_i^\ast = \Ext_i L_i \Ext_i^\ast$. Then we use Proposition~\ref{Prop: f(L) decomposition} for $f(z)\coloneqq \sqrt{z}$ and the compatibility of the functional calculus with unitary equivalences in Proposition~\ref{Prop: FC transformation}.(ii) to figure out that
\begin{align}
		\dom(\sqrt{\bL})
		=S^{-1}\Bigl(\bigotimes_i \dom(\sqrt{\bL \Pro_i^\ast}) \Bigr)  
		= S^{-1}\Bigl( \bigotimes_i \Ext_i \dom(\sqrt{L_i}) \Bigr).
\end{align}
On the other hand, we conclude from \eqref{W12D desired splitting} that
\begin{align}
\W^{1,2}_{\bD}(\bO)^m = S^{-1} \Bigl( \bigotimes_i \Pro_i \W^{1,2}_{\bD}(\bO)^m \Bigr) = S^{-1}\Bigl( \bigotimes_i \Ext_i \W^{1,2}_{D_i}(O_i)^m \Bigl).
\end{align}
Both chains of equalities are topological. Now, (ii) is the same as saying that the tensor spaces on the right-hand sides coincide up to equivalent norms. Since $S$ is an isomorphism, this is equivalent to (i).
\end{proof}

\subsection{Embedding of $O$ into an interior thick set}
\label{Subsec: fat set embedding}

Now, we reverse the order of reasoning by embedding the geometric configuration $(O,D)$ in a \enquote{fattened version} $(\bO, \bD)$ with the same geometric quality but additionally satisfying the interior thickness condition. This is the content of the following proposition, the proof of which will occupy the rest of the section.

\begin{proposition} \label{Prop: fat set embedding} 
	Let $O$ and $D$ be as in Theorem~\ref{Thm: main kato result} and put $N\coloneqq \bd O \setminus D$. Then there exists an open interior thick set $\bO \supseteq O$ such that $\bO \setminus O$ is open and $\bd O \subseteq \bd \bO$. With $\bD\coloneqq \bd \bO \setminus N$ one has that $\bD$ is closed and Ahlfors--David regular, $D\subseteq \bD$ and $\bO$ is locally uniform near $\bN\coloneqq N$. In particular, $\bD \cap \bd O = D$.
\end{proposition}

By assumption $O$ is locally an $(\eps,\delta)$-domain near $N$. Let $\Sigma$ denote a grid of open axis-parallel cubes of diameter $\frac{\delta}{8}$ in $\R^d$. Let $\Sigma'$ contain those cubes $Q$ from $\Sigma$ for which $\overline{Q}$ intersects $D$ but which stay away from $N$, say $Q\cap N_{\delta/4} = \emptyset$ for good measure. Now, put
\begin{align}
	\bO \coloneqq O \cup \bigcup_{Q\in \Sigma'} (Q\setminus \bd O) \quad \text{and} \quad \bD \coloneqq \bd \bO \setminus N.
\end{align} 
Clearly $\bO$ is an open superset of $O$. Moreover, $\bD$ is a closed superset of $D$ because $\bO \setminus O$ stays away from $N$ and hence the relative openness of $N$ in $\bd O$ is inherited to $\bd \bO$. We have to verify the following conditions:
\begin{enumerate}[(a)]
	\item $\bO \setminus O$ is open and $\bd O \subseteq \bd \bO$, \label{fat prop 1}
	\item $\bO$ is locally uniform near $\bN$, \label{fat prop 2}
	\item $\bO$ is interior thick, that is, it satisfies~\eqref{ITC}, and \label{fat prop 3}
	\item $\bD$ satisfies the Ahlfors--David condition. \label{fat prop 4}
\end{enumerate}
Then $D \subseteq \bD \subseteq {}^c N$ implies $\bD \cap \bd O = D$.

\textbf{Proof of \ref{fat prop 1}.}
We have $\bO \setminus O = \bigcup_{Q \in \Sigma'} (Q \setminus \cl{O})$, which is an open set. Since $O$ is open as well, $\bd O \subseteq \bd \bO$ follows. \hfill $\square$ \\

\textbf{Proof of \ref{fat prop 2}.}
From the construction of $\Sigma'$ we get $\bO \cap N_{\delta/4} = O \cap N_{\delta/4}$. This already gives the $\eps$-cigar condition with $\delta$ replaced by $\delta/4$. Since $O$ is open and closed in $\bO$, connected components of $O$ are also connected components of $\bO$ and hence satisfy the positive radius condition by assumption. All remaining connected components keep distance to $N$ and are therefore not considered in the positive radius condition. Consequently, $\bO$ is an $(\eps, \frac{\delta}{4})$-domain near $\bN = N$. \hfill $\square$ \\

The boundary $\bd O$ is porous by Corollary~\ref{Cor: Boundary porous}. This implies $|\bd O | = 0$, see \cite[Lem.~A.1]{BE}. Since the cubes in $\Sigma$ are all of the same size, we can also record the following observation. We will freely use these facts from now on.

\begin{lemma}
\label{Lem: A cubes is a good thing}
Each cube $Q \in \Sigma$ is interior thick and has Ahlfors--David regular boundary, where implicit constants depend only on $\delta$ and $d$.
\end{lemma}

\textbf{Proof of \ref{fat prop 3}.}
Let $B$ be a ball of radius $r=\rad(B)\leq 1$ with center $x\in \bO$. If $x\in \cl{Q}$ for some $Q\in \Sigma'$, then $|B \cap \bO|\geq|B \cap Q| \gtrsim r^d$ with implicit constant depending on $\delta$ and $d$. Otherwise, we must have $x\in O$. If additionally $x\in N_{\delta/2}$, then Proposition~\ref{Prop: locally eps-delta yields corkscrew} yields the desired lower bound $|B \cap \bO| \gtrsim r^d$.

It remains to treat the case $x \in O \setminus N_{\delta/2}$. Let $Q'$ be a cube in the grid $\Sigma$ whose closure contains $x$. Again, if $Q' \subseteq O$, then $|B \cap \bO| \geq |B\cap Q'|$ and we are done. If not, then $Q'$ intersects ${}^cO$ and from $x \in O$ we conclude that $Q'$ contains some $z\in \bd O$. By the size of $Q'$ we infer $Q'\cap N_{\delta/4} = \emptyset$. Therefore, we have $z\in D$, which implies $Q'\in \Sigma'$ and we are back in the very first case. \hfill $\square$ \\

We continue with the following 

\begin{lemma} \label{Lem: bD}
	One has $\bD = D \cup \bigcup_{Q\in \Sigma'} (\bd Q\setminus O)$.
\end{lemma}
\begin{proof}
	We show that
	\begin{align}
	\bd \bO = \bd O \cup \bigcup_{Q\in \Sigma'} (\bd Q \setminus O)
	\end{align}
	since then the lemma follows by intersection with ${}^cN$, taking into account $\dist(Q,N) \geq \frac{\delta}{4}$ for $Q \in \Sigma'$. 

	Let $x\in \bd \bO$. Then $x\not\in O$ since $\bO$ is open and contains $O$. If every neighborhood of $x$ intersects $O$, then $x\in \bd O$.
	
	Otherwise, there is some open ball $B$ with center $x$ that is disjoint to $O$. However, as every neighborhood of $x$ intersects $\bO$ by assumption, there must be some $Q '\in \Sigma'$ such that $B\cap Q' \neq \emptyset$. Since $\Sigma$ is a grid, the ball $B$ only intersects finitely many cubes in $\Sigma'$. We conclude that the sequence of balls $(\frac{1}{n} B)_{n \in \IN}$ hits some cube $Q \in \Sigma'$ infinitely often. Thus, $x$ is in the closure of $Q$. But $x \notin \bO $ and $x \notin \cl{O}$ imply $x \notin Q$, hence we must have $x\in \bd Q \setminus O$. This completes the proof of the inclusion ``$\subseteq$''.
	
	Conversely, let $x \in\bd O \cup \bigcup_{Q\in \Sigma'} (\bd Q\setminus O)$. If $x\in \bd O$, then $x \in \bd \bO$ follows from \ref{fat prop 1}. Otherwise, there is some $Q' \in \Sigma'$ such that $x\in \bd Q' \setminus O$. Since $\Sigma$ is a grid of open cubes, this implies $x \notin Q$ for every $Q \in \Sigma$. Hence, we have $x\not\in \bO$. But each neighborhood of $x$ intersects $Q'$ and since the boundary of an open set has no interior points, it also intersects $Q' \setminus \bd O\subseteq \bO$. Hence, we have $x\in \bd \bO$.
\end{proof}

\textbf{Proof of \ref{fat prop 4}.}
 For the rest of the section let $B$ be some ball with center $x$ in $\bD$ and radius $r=\rad(B)\in (0,\diam(\bD))$. Our task is to show comparability
\begin{align}
\label{eq: ADR for bD}
\cH^{d-1}(\bD \cap B) \approx r^{d-1}
\end{align}
and we organize the argument in the cases coming from Lemma~\ref{Lem: bD}. By the same lemma we have $\diam(\bD) =\infty$ if and only if we have $\diam(D) = \infty$. Hence, we can use the Ahlfors–-David condition for $D$ with balls up to radius say $2\diam(\bD)$, compare with Remark~\ref{Rem: AD}.

\emph{Case 1: $x\in D$}. The lower bound in \eqref{eq: ADR for bD} follows directly from the Ahlfors--David condition for $D$: 
\begin{align}
 \cH^{d-1}(\bD \cap B) \geq \cH^{d-1}(D\cap B) \gtrsim r^{d-1}.
\end{align}
For the upper bound we need to make sure that $B$ does not intersect too many cubes in $\Sigma'$. Consider the subcollection
\begin{align}
\label{eq1: ADR for bD}
	\Sigma^\prime_B \coloneqq \{ Q\in \Sigma': \cl{Q} \cap B \neq \emptyset \}.
\end{align}
If $r\leq \delta$, then $\# \Sigma^\prime_B \lesssim 1$ by the grid size of $\Sigma$ and we obtain
\begin{align}
\label{eq1: fat property 4}
\cH^{d-1}\bigg(\bigcup_{Q\in \Sigma'} \bd Q \cap B \bigg) \leq \sum_{Q\in \Sigma^\prime_B} \cH^{d-1}(\bd Q \cap B) \lesssim (\# \Sigma^\prime_B) r^{d-1} \lesssim r^{d-1},
\end{align}
where the third step uses the upper bound in the Ahlfors--David condition for the boundaries of the cubes $Q \in \Sigma$ as in Lemma~\ref{Lem: A cubes is a good thing}, even though $B$ is not necessarily centered in $\bd Q$. This is not an issue because if $\bd Q \cap B$ is not empty, then $B$ is contained in a ball with doubled radius centered in $\bd Q$.

If $r>\delta$, then we need the following lemma to bound the size of $\Sigma^\prime_B$. Its proof is similar to that of Lemma~\ref{Lem: bdO lemma from EHT}.
\begin{lemma}
\label{Lem: strip lemma fattening}
	Let $x \in D$, $r\in (0,2\diam(\bD))$, $h\in (0,r)$ and consider the set $$E_{r,h} \coloneqq \{y\in \B(x,r): \dist(y, D) \leq h \}.$$ Then $|E_{r,h}| \lesssim h r^{d-1}$, where the implicit constant only depends on $D$ and $d$.
\end{lemma} 
\begin{proof}
	For convenience, put $E\coloneqq E_{r,h}$ and fix $y'\in E$. Associate to each $y\in E$ some $y_D\in D$ with $|y-y_D| \leq h$. We claim that $B\coloneqq \B(y'_D, 4r) \supseteq \{y_D: y\in E\}$. Indeed, $$|y_D-y'_D| \leq |y_D-y|+|y-y'|+|y'-y'_D| < h + 2r + h < 4r.$$ Moreover, there is some $C>0$ such that $B\cap D$ can be covered by $C(r/h)^{d-1}$ many balls $B_i$ of radius $h$ centered in $D$, see~\cite[Lemma A.5]{BE}. Next, pick $y\in E$. Then $y_D \in B\cap D$ and therefore $y_D\in B_i$ for some $i\in I$. Thus, $y\in 2 B_i$ and consequently $E\subseteq \bigcup_i 2 B_i$. Finally,
	\begin{align}
		|E|\leq \sum_i |2B_i| \lesssim \#_i h^d \lesssim (r/h)^{d-1} h^d = h r^{d-1}. &\qedhere
	\end{align}
\end{proof}

Coming back to finding a substitute for \eqref{eq1: fat property 4} in the case $r> \delta$, we claim that $\bigcup \Sigma^\prime_B \subseteq E_{2r, \delta/8}$.
Indeed, let $Q\in \Sigma^\prime_B$ and $y\in Q$. Since $\overline{Q}$ intersects $D$, we have $\d(y,D) \leq \diam(Q) = \frac{\delta}{8}$, and by definition of $\Sigma^\prime_B$ in \eqref{eq1: ADR for bD} there is some $z\in \cl{Q}\cap B$ so that $$|x-y| \leq |x-z|+|z-y| \leq r + \diam(Q) \leq r + \frac{\delta}{8} < 2 r.$$ Owing to Lemma~\ref{Lem: strip lemma fattening}, we can now do the following counting argument:
\begin{align}
	(\#\Sigma^\prime_B) \delta^d \approx \bigg|\bigcup \Sigma^\prime_B \bigg| \leq |E_{2r, \delta/8}| \lesssim \delta r^{d-1}.
\end{align}
It follows that $\#\Sigma^\prime_B \lesssim (r/\delta)^{d-1}$, which in turn gives
\begin{align}
	\cH^{d-1}\bigg(\bigcup_{Q\in \Sigma'} \bd Q \cap B\bigg) 
	\leq \sum_{Q \in \Sigma'_B} \cH^{d-1}(\bd Q)
	\approx (\#\Sigma^\prime_B) \delta^{d-1} \lesssim r^{d-1}.
\end{align}
This is the same upper bound as in the case $r< \delta$, see \eqref{eq1: fat property 4}. Hence, in both cases the Ahlfors--David condition for $D$ allows us to estimate
\begin{align}
	\cH^{d-1}(\bD \cap B) \leq \cH^{d-1}(D\cap B) + \cH^{d-1} \bigg(\bigcup_{Q\in \Sigma'} \bd Q \cap B \bigg) \lesssim r^{d-1},
\end{align}
which gives the required upper bound in \eqref{eq: ADR for bD}.

\emph{Case 2: $x\in \bd Q \setminus O$ for some $Q\in \Sigma'$}. We distinguish whether $\frac{1}{2} B$ is disjoint to $D$ or not. So, let us first suppose that $\frac{1}{2} B \cap D \neq \emptyset$ and let $z\in \frac{1}{2} B \cap D$. Then $\B(z, \frac{r}{2})$ is centered in $D$ and contained in $B$, so that we obtain the lower bound from the Ahlfors--David condition for $D$ applied to $\B(z, \frac{r}{2})$. For the upper bound, we use Case~1 applied to $\B(z,2r) \supseteq B$.

Now, suppose that $\frac{1}{2} B \cap D = \emptyset$. By construction of $\Sigma'$ we have $\dist(x,D) \leq \frac{\delta}{8}$ and $\dist(x,N) \geq \frac{\delta}{4}$. Hence $\frac{r}{2} \leq \frac{\delta}{8}$ and $\frac{1}{2}B$ does not intersect $\bd O = D \cup N$. Since this ball is centered outside of $O$ and does not intersect $\bd O$, we must have $\frac{1}{2} B \subseteq {}^c O$. This shows $\bd Q\cap \frac{1}{2} B = (\bd Q \setminus O) \cap \frac{1}{2} B$ and the lower bound in \eqref{eq: ADR for bD} follows from the Ahlfors--David condition for $\bd Q$. For the upper bound we argue as in Case~1 with radii $r \leq \delta$. 

This concludes the proof of \ref{fat prop 4} and hence the proof of Proposition~\ref{Prop: fat set embedding}. \hfill $\square$ \\

\subsection{Proof of Theorem~\ref{Thm: main kato result}}
\label{Subsec: Proof of second main result}

We combine Theorem~\ref{Prop: Regularity decomposition for the square root} with Proposition~\ref{Prop: fat set embedding}.

	Given $O$ and $D$, construct sets $\bO$ and $\bD$ according to Proposition~\ref{Prop: fat set embedding}. To ease the connection with Section~\ref{Sec: Projections coming from indicator functions} we put $O_0\coloneqq O$ and $O_1\coloneqq \bO \setminus O$. Then $D_0 \coloneqq \bD \cap \bd O = D$ and $D_1 \coloneqq \bD \cap \bd O_1 = \bd O_1$. We extend the coefficients $A$, $b$, $c$, $d$ to coefficients $\boldsymbol{A}$, $\boldsymbol{b}$, $\boldsymbol{c}$, $\boldsymbol{d}$ on $\bO$. For $\boldsymbol{A}$ and $\boldsymbol{d}$ we put the corresponding identity matrix on $\bO \setminus O$ to ensure ellipticity and $\boldsymbol{b}$, $\boldsymbol{c}$ are simply extended by zero. With those extended coefficients we define the operator $\bL$ with form domain $\W^{1,2}_{\bD}(\bO)$ as in~\eqref{eq: L as composition}. Since by \eqref{W12D desired splitting} the form domain for $\bL$ splits as
	\begin{align*}
	\W^{1,2}_{\bD}(\bO) \cong 	\W^{1,2}_{D_0}(O_0) \otimes \W^{1,2}_{D_1}(O_1),
	\end{align*}
	we get that the coefficients for $\bL$ are again elliptic in the sense of \eqref{Garding} with lower bound $\min(\lambda,1)$ and upper bound $\max(\Lambda,1)$. Here, $\Lambda$ is an upper bound for the coefficients of $L$.  Proposition~\ref{Prop: similarity L_i and part of L} reveals that $L$ is unitarily equivalent to $\bL \Pro_0^\ast$.
	It was shown in Section~\ref{Sec: AKM} that Theorem~\ref{Thm: main kato result} is valid on interior thick sets. Consequently, we can apply that theorem to the operator $\bL$ on $\bO$. But this brings us into the business of Proposition~\ref{Prop: Regularity decomposition for the square root} and we can conclude the square root property for~$L$. \hfill $\square$ \\
\appendix
\section{Functional calculus for (bi)sectorial operators}
\label{Sec: Functional calculus for (bi)sectorial operators}

We provide the essentials on functional calculi for (bi)sectorial operators that are needed for understanding our paper. The reader who is not familiar with this theory can consult~\cite{McIntosh-Hinfty, Haase} and also \cite{Diss, Hytonen-Vol-II} for the bisectorial case.

For $\omega \in (0, \pi)$ let $\Sec^+_\omega \coloneqq \{z \in \IC \setminus \{0\} : |\arg z| < \omega \}$ be the sector of opening angle $2 \omega$ symmetric about the positive real axis. It will be convenient to set $\Sec_0^+ \coloneqq (0,\infty)$. The bisector of angle $\omega \in [0,\pi/2)$ is defined by $\Sec_{\omega} \coloneqq \Sec^+_{\omega} \cup (-\Sec^+_{\omega})$. 

\begin{definition}
\label{Def: (Bi)sectorial operators}
A linear operator $T$ in a Hilbert space $H$ is \emph{sectorial} of angle $\omega \in [0, \pi)$ if its spectrum $\sigma(T)$ is contained in $\cl{\Sec^+_\omega}$ and if 
\begin{align}
\label{eq: Def bisectorial}
\IC \setminus \cl{\Sec^+_\varphi} \to \cL(H), \qquad z  \mapsto z(z-T)^{-1}
\end{align}
is uniformly bounded for every $\varphi \in (\omega, \pi)$.  \emph{Bisectorial} operators of angle $\omega \in [0, \pi/2)$ are defined similarly upon replacing sectors by bisectors.
\end{definition}

Such operators are automatically closed and densely defined~\cite[Prop.~2.1.1]{Haase}.

\begin{example}
\label{Ex: Bisectorial operators}
Self-adjoint operators are bisectorial of angle $0$, see \cite[Prop.~C.4.2]{Haase}. An operator $T$ is called \emph{maximal accretive} if it satisfies $\|(z+T)^{-1}\| \leq (\Re z)^{-1}$ for all $z \in \IC$ with $\Re z > 0$. This implies sectoriality of angle~$\pi/2$.
\end{example}

Throughout, we denote by $\cM(U)$ and $\H^\infty(U)$ the meromorphic and bounded holomorphic functions on an open set $U\subseteq \IC$, respectively. 

Let $T$ be sectorial of angle $\omega$ and let $\varphi \in (\omega,\pi)$. The construction of its functional calculus starts from the subalgebra $\H^\infty_0(\Sec^+_\varphi)$ of functions $f \in \H^\infty(\Sec^+_\varphi)$ satisfying $|f(z)| \leq C\min(|z|^\alpha, |z|^{-\alpha})$ for some $C,\alpha > 0$ and all $z \in \Sec_\varphi^+$. In this case fix $\nu \in (\omega,\varphi)$ and define
\begin{align}
	f(T)\coloneqq \frac{1}{2\pi i} \int_\gamma f(z)(z-T)^{-1} \d z,
\end{align}
where $\gamma$ is a positively oriented parametrization of $\bd \Sec^+_\nu$. This yields an algebra homomorphism $\H^\infty_0(\Sec^+_\varphi) \to \cL(H)$. The canonical extension to the subalgebra 
\begin{align*}
\cE(\Sec^+_\varphi) \coloneqq \H^\infty_0(\Sec^+_\varphi) \oplus \langle (1+z)^{-1} \rangle \oplus \langle 1 \rangle
\end{align*}
of $\cM(\Sec^+_\varphi)$ is well-defined and again an algebra homomorphism. By regularization this algebra homomorphism extends to an unbounded functional calculus within $\cM(\Sec^+_\varphi)$ as follows. Introduce the algebra
\begin{align}
	\cM(\Sec^+_\varphi)_T \coloneqq \Bigl\{ f\in \cM(\Sec^+_\varphi): \exists \, e\in \cE(\Sec^+_\varphi) \text{ such that } e(T) \text{ is injective and } ef\in \cE(\Sec^+_\varphi) \Bigr\}
\end{align}
and define the closed operator $f(T)$ for $f\in \cM(\Sec^+_\varphi)_T$ by $f(T)\coloneqq e(T)^{-1} (ef)(T)$. This definition does not depend on the choice of the regularizer $e$. The reasonable generalization of the notion of algebra homomorphisms in this context is to have $f(T)+g(T)\subseteq (f+g)(T)$ and $f(T)g(T)\subseteq (fg)(T)$ with equality if $f(T)$ is bounded. Indeed, this is the case for our construction~\cite[Thm.~1.3.2]{Haase}.

\begin{example}
\label{Ex: Fractional powers}
The \emph{fractional powers} $T^\alpha$ for $\Re \alpha > 0$ are defined via the regularizer $(1+z)^{-n}$, where $n$ is an integer larger than $\Re \alpha$. They satisfy the law of exponents $T^\alpha T^\beta = T^{\alpha + \beta}$ and if $T$ is invertible, then so is $T^\alpha$. See~\cite[Prop.~3.1.1]{Haase} for these properties.
\end{example}

\begin{example}
\label{Ex: Hinfty}
Let $f \in \H^\infty(\Sec_\varphi^+)$. If $T$ is injective, then $z(1+z)^{-2}$ regularizes $f$. In general, there is a topological splitting $H = \Ke(T) \oplus \cl{\Rg(T)}$ and the part of $T$ in $\cl{\Rg(T)}$ is an injective sectorial operator of the same angle~\cite[p.~24]{Haase}. Hence, $f(T)$ is always defined as a closed operator in $\cl{\Rg(T)}$ via $f(T) \coloneqq f(T|_{\cl{\Rg(T)}})$.
\end{example}

The calculus for bisectorial operators is constructed in the same way upon systematically replacing sectors by bisectors and $(1+z)^{-1}$ by $(\i+z)^{-1}$. It shares the same properties except that instead of fractional powers one rather considers $(T^2)^\alpha$ for $\Re \alpha > 0$. No ambiguity can occur when writing down such expressions. Indeed, if $T$ is bisectorial, then $T^2$ is sectorial -- and if $f(T^2)$ is defined by the sectorial calculus, then $f(T^2) = (f\circ z^2)(T)$, see \cite[Thm.~3.2.20]{Diss}.

We need the following transformation properties. We include a proof since we could not find the precise statement of (i) in the literature.

\begin{proposition}
\label{Prop: FC transformation}
	Let $T$ be a (bi)sectorial operator in a Hilbert space $H$.
	
	\begin{enumerate}
	\item Let $\Pro$ be an orthogonal projection in $H$ and let $\Pro^\ast: \Pro H \to H$ the inclusion map. Suppose that $\Pro T\subseteq T\Pro$. Then $T\Pro^\ast$ is a (bi)sectorial operator in $\Pro H$ with the same angle and one has
	\begin{align}
		f(T\Pro^\ast)=f(T)\Pro^\ast \quad\text{and}\quad \Pro f(T) \subseteq f(T)\Pro
	\end{align}
	for every $f$ in the functional calculus for $T$.
	
	\item Let $S: H \to K$ be an isomorphism onto another Hilbert space. Then $S^{-1} T S$ is again (bi)sectorial in $K$ with the same angle. It has the same algebra of admissible functions $f$ as $T$ and $f(S^{-1}TS) = S^{-1} f(T) S$ with $\dom(f(S^{-1}TS)) = S^{-1} \dom(f(T))$ holds.
	\end{enumerate}
\end{proposition}
\begin{proof}
	We begin with the proof of (i).
	
	\emph{Step 1: $f = (\lambda-z)^{-1}$ with $\lambda \in \rho(T)$}. The assumption implies $\Pro (\lambda - T) \subseteq (\lambda - T) \Pro$ and hence $\Pro (\lambda - T)^{-1} = (\lambda - T)^{-1} \Pro$. On $\Pro \dom(T)$ we have $\lambda - T = \lambda - T \Pro^\ast$ and on $\Pro H$ we have $(\lambda - T)^{-1} = (\lambda - T)^{-1} \Pro^\ast$. With this at hand, a direct calculation shows that $(\lambda - T)^{-1} \Pro^\ast$ is an operator on $\Pro H$ that acts as a two-sided inverse for $\lambda - T \Pro^\ast$, that is to say, $(\lambda-T\Pro^\ast)^{-1}=(\lambda-T)^{-1}\Pro^\ast$. In particular, $T \Pro^\ast $ is a (bi)sectorial operator in $\Pro H$.
	
	\emph{Step 2}: $f\in \cE(\Sec^+_\varphi)$. Both assertions follow directly by construction and Step 1.
	
	\emph{Step 3}: $f\in \cM(\Sec^+_\varphi)_T$. Let $e$ be a regularizer. By Step 2 we have $e(T\Pro^\ast)=e(T)\Pro^\ast$ and this operator is clearly injective. Then
	\begin{align}
		f(T\Pro^\ast)=e(T\Pro^\ast)^{-1} (ef)(T\Pro^\ast) = (e(T)\Pro^\ast)^{-1}(ef)(T)\Pro^\ast.
	\end{align}
	But this reduces directly to $e(T)^{-1}(ef)(T)\Pro^\ast=f(T)\Pro^\ast$ since $e(T)$ and $(ef)(T)$ preserve $\Pro H$. Moreover, we obtain $\Pro f(T)\subseteq f(T) \Pro$ by the respective inclusions for $e(T)^{-1}$ and $(ef)(T)$ from Step 2.
	
	In the proof of (ii) we directly have $(\lambda - S^{-1}TS)^{-1} = S^{-1}(\lambda - T) S$ for $\lambda \in \rho(T)$ and the rest of the proof follows the same pattern as above.
\end{proof}

\begin{definition}
\label{Def: Hinfty calculus}
Let $T$ be a bisectorial operator of angle $\omega$ in a Hilbert space $H$ and let $\varphi \in (\omega, \pi/2)$. If $f(T)$ is a bounded operator on $\overline{\Rg(T)}$ for all $f \in \H^\infty(\Sec_\varphi)$ and there is a constant $C$ such that the operator norm estimate
\begin{align*}
\|f(T)\|_{\cl{\Rg(T)} \to \cl{\Rg(T)}} \leq C \|f\|_\infty \qquad \mathrlap{(f \in \H^\infty(\Sec_\varphi))}
\end{align*}
holds, then $T$ is said to have a \emph{bounded $\H^\infty$-calculus} of angle $\varphi$ with bound $C$ on $\cl{\Rg(T)}$.
\end{definition}

The following fundamental theorem of McIntosh~\cite{McIntosh-Hinfty} characterizes this property through \emph{quadratic estimates}, see also \cite[Thm.~3.4.11 \& Cor.~3.4.14]{Diss}. 
\begin{theorem}
\label{Thm: McIntosh}
Let $T$ be a bisectorial operator of angle $\omega$ in a Hilbert space $H$ and let $\varphi \in (\omega,\pi/2)$. Then $T$ has a bounded $\H^\infty$-calculus of angle $\varphi$ if and only if
\begin{align}
\label{eq: McIntosh}
\int_0^\infty \|tT(1+t^2T^2)^{-1} u\|_H^2 \; \frac{\d t}{t} \approx \|u\|_H^2 \qquad \mathrlap{(u \in \cl{\Rg(T)}).}
\end{align}
The bound for the $\H^\infty$-calculus depends on $\varphi$ and the implicit constants in \eqref{eq: McIntosh} and \eqref{eq: Def bisectorial}.
\end{theorem}

\begin{example}
\label{Ex: sa implies Hinfty}
Self-adjoint operators have a bounded $\H^\infty$-calculus by the spectral theorem. One can also check quadratic estimates in an elementary manner~\cite[Ex.~3.4.15]{Diss}.
\end{example}

Bisectorial operators with a bounded $\H^\infty$-calculus satisfy the following abstract square root estimate~\cite[Prop.~3.3.15]{Diss}. This essentially follows from considering the bounded operator $(z/\sqrt{z^2})(T)$ and its inverse on $\cl{\Rg(T)}$. Hence, the norm bounds are
explicit.
\begin{example}
\label{Ex: Abstract Kato}
Let $T$ be a bisectorial operator in a Hilbert space $H$ with a bounded $\H^\infty$-calculus of angle $\varphi$ and bound $C$. It follows that $\dom(\sqrt{T^2}) = \dom(T)$ with comparability
\begin{align*}
C^{-1} \|Tu\|_H \leq \|\sqrt{T^2}u\|_H \leq C \|Tu\|_H \qquad \mathrlap{(u \in \dom(T)).}
\end{align*}
\end{example}

\end{document}